\documentclass[11pt]{article}
\usepackage{myarticle-macrosShared}
\usepackage{graphicx,subfig,enumerate,bm}
\usepackage{geometry}
\usepackage{color}

\usepackage{setspace}
\date{November 11th, 2013}

\newcommand{\gammaHat}{\widehat{\gamma}}

\newcommand{\rTilde}{\tilde{r}}

\newcommand{\polyLog}{\textrm{polyLog}}
\newcommand{\DPsiPrimeOneLessPred}{D_{\psi'(r_{\cdot,[p]})}}
\newcommand{\outerExp}[1]{\mathbf{E}^*\left(#1\right)}
\title{Asymptotic behavior of unregularized and ridge-regularized high-dimensional robust regression estimators : rigorous results}
\author{Noureddine El Karoui\\ Department of Statistics, UC Berkeley\thanks{I would like to thank Peter Bickel, Bin Yu and Derek Bean for many interesting discussions on high-dimensional robust regression estimators. I am especially grateful to Peter Bickel for many fascinating discussions that greatly improved my understanding of this topic. 
Support from NSF grant
DMS-0847647 (CAREER) is gratefully acknowledged.
\textbf{AMS 2010 MSC: } Primary: 62E20. Secondary: 60F99 \textbf{Key
words and Phrases~: } high-dimensional inference, random matrix theory, concentration of measure, proximal mapping, regression M-estimates, robust regression.
\textbf{Contact~:} \texttt{nkaroui@berkeley.edu}}}
\begin{document}
\maketitle
\begin{abstract}
We study the behavior of high-dimensional robust regression estimators in the asymptotic regime where $p/n$ tends to a finite non-zero limit. More specifically, we study ridge-regularized estimators, i.e 
$$
\betaHat=\argmin_{\beta \in \mathbb{R}^p} \frac{1}{n}\sum_{i=1}^n \rho(\eps_i-X_i\trsp \beta)+\frac{\tau}{2}\norm{\beta}^2\;.
$$
In a recently published paper, we had developed with collaborators probabilistic heuristics to understand the asymptotic behavior of $\betaHat$. We give here a rigorous proof, properly justifying all the arguments we had given in that paper. Our proof is based on the probabilistic heuristics we had developed, and hence ideas from random matrix theory, measure concentration and convex analysis. 

While most the work is done for $\tau>0$, we show that under some extra assumptions on $\rho$, it is possible to recover the case $\tau=0$ as a limiting case.

We require that the $X_i$'s be i.i.d with independent entries, but our proof handles the case where these entries are not Gaussian. 

A 2-week old paper of Donoho and Montanari studied a similar problem by a different method and with a different point of view. At this point, their interesting approach requires Gaussianity of the design matrix. 
\end{abstract}
\section{Introduction}
Robust regression estimators, also known as regression $M$-estimates, have been of interest in Statistics for at least the last five decades. They are natural extensions of the least-squares problem: namely we estimate a regression vector by solving the optimization problem
\begin{equation}\label{eq:defRegressionMEstimate}
\betaHat=\argmin_{\beta \in \mathbb{R}^p} \frac{1}{n}\sum_{i=1}^n \rho(Y_i-X_i\trsp \beta)\;.
\end{equation}
Here, $X_i\in \mathbb{R}^p$ is a vector of predictors and $Y_i\in \mathbb{R}$ is a scalar response. $\rho$ is a function from $\mathbb{R}$ to $\mathbb{R}$. Typically once assumes that there is a linear relationship between $X_i$ and $Y_i$, i.e 
$$
Y_i=X_i\trsp \beta_0+\eps_i\;,
$$
where $\eps_i$ are considered to be unknown random errors, and $\beta_0$ is an unknown fixed vector one wishes to estimate. The $n\times p$ matrix $X$ whose $i$-th row is $X_i\trsp$ is called the design matrix.

Huber's papers from the 1970's (\cite{HuberWaldLecture72}, \cite{HuberRobustRegressionAsymptoticsETCAoS73}) contain a number of very interesting results, including limiting behavior for $\betaHat$ as $n\tendsto \infty$ when $p$ is held fixed. Huber also raised the question of understanding the behavior of the estimators when $p$ is large and obtained partial results in the least-squares case. Further interesting contributions happened in the mid to late 80's with work of Portnoy (\cite{PortnoyMestLargishPNConsistencyAoS84}, \cite{PortnoyMestLargishPNCLTAoS85}, \cite{PortnoyCLTRobustRegressionJMVA87}) and Mammen (\cite{MammenRobustRegressionAos89}). In these studies, the authors studied the behavior of regression M-estimates when $p$ and $n$ are both large, but $p/n\tendsto 0$ at various rates. Some of the papers refer to fixed design (i.e $X$ is non-random and the only source of randomness in the problem are $\eps_i$'s), others treat the random design case (i.e both $X$ and $\eps_i$ are random). 

A central result of Huber (see e.g \cite{HuberRonchettiRobustStatistics09}) is that when $p$ is held fixed, and $\eps_i$'s are i.i.d, the optimal $\rho$ one can use is $-\log f_\eps$, where $f_\eps$ is the density of the errors - at least when one measures quality of the estimator by the size of $\scov{\betaHat}$. In \cite{NEKetAlRobustRegressionTechReport11}, \cite{NEKRobustPaperPNAS2013Published} a group of us looked at corresponding questions in the high-dimensional setting where $p/n$ is not small and found the situation to be very different. Indeed, it was clear that one could do better than using $-\log f_\eps$. In \cite{NEKetAlRobustRegressionTechReport11}, we proposed a probabilistic heuristic to understand the behavior of $\betaHat$ and verified the quality of its predictions on several simulations and computations. Our heuristic led to the formulation of a natural variational problem, which we solved in \cite{NEKOptimalMEstimationPNASPublished2013}. Interestingly, the solution of the variational problem depends in general on $p/n$, i.e the dimensionality of the problem. (\cite{NEKetAlRobustRegressionTechReport11} is the long form of the paper \cite{NEKRobustPaperPNAS2013Published}, which is very short  due to page-limit requirements.)

Our heuristic is based on random matrix and concentration of measure ideas. We prove in this paper that these ideas can be used rigorously and indeed, under various assumptions, rigorously justify the claims made in \cite{NEKetAlRobustRegressionTechReport11} and \cite{NEKRobustPaperPNAS2013Published}.

The assumptions under which we operate for the design matrix reflect the central role played by the concentration of measure phenomenon (\cite{ledoux2001}) in this problem. 

A couple of weeks ago, Donoho and Montanari (\cite{DonohoMontanariRobustArxiv13}) announced a proof of some of the results explained in \cite{NEKRobustPaperPNAS2013Published} under the assumption that the design matrix is full of i.i.d Gaussian random variables (i.e $X_i$'s are independent with i.i.d Gaussian entries). Their proof uses different ideas than ours - it is based on the technology of rigorous analysis of approximate message passing algorithms (see \cite{DonohoMalekiMontanariAMP09PNAS} and \cite{BayatiMontanariLASSORISKIEEE2012}). 

By working under concentration assumptions, we are able to show a number of the same results without requiring i.i.d-ness of the entries of the vectors $X_i$'s. However, to prove the main result, we still need the $X_i$'s to have i.i.d entries, but they do not need to be Gaussian. Donoho and Montanari also make interesting connections with rigorous work in statistical physics, namely to the so-called Shcherbina-Tirozzi model (\cite{ShcherbinaTirozzi03} and \cite{TalagrandSpinGlassesBook03}) and other heuristic approaches based on approximate message passing (\cite{Rangan2011}).

Our proof also makes rigorous the probabilistic heuristics that were developed in \cite{NEKRobustPaperPNAS2013Published}. Our point of view is that the properties of $\betaHat$ defined in Equation \eqref{eq:defRegressionMEstimate} via connections to random matrix theory. As such, our proof relies heavily on leave-one-out, martingale and concentration of measure ideas, as some of our previous work (see e.g \cite{nekCorrEllipD}) did in establishing these connections. Leave-one-out ideas seem to be known in Physics under the name ``cavity method", so our general approach falls broadly in that category. A number of the tools we use are commonly used in the spectral analysis of large random matrices via the Stieltjes transform method (see \cite{mp67}, \cite{wachter78}, \cite{silverstein95}). 

\subsection{Focus of the paper}

We focus on the problem of understanding
\begin{equation}\label{eq:defBetaHatRidgeregularizedCase}
\betaHat=\argmin_{\beta \in \mathbb{R}^p} \frac{1}{n}\sum_{i=1}^n \rho(\eps_i-X_i\trsp \beta)+\frac{\tau}{2}\norm{\beta}^2
\end{equation}
where $\tau>0$. We will see later (see Section \ref{sec:regularizedToUnregularized}) that under certain conditions on $\rho$ the understanding of $\betaHat$ for various $\tau$'s will lead us to rigorous understanding of $\betaHat$ when $\tau=0$. 

Different parts of the proof require different assumptions. So we label the assumptions accordingly.

For the first part of the proof (i.e ``leave-one-observation-out''), we work under the following assumptions:
\begin{itemize}
\item \textbf{O1}: $p/n$ has a finite non-zero limit.
\item \textbf{O2}: $\rho$ is twice differentiable, convex and non-linear. $\psi=\rho'$. Note that $\psi'\geq 0$ since $\rho$ is convex. We assume that $\rho\geq 0$ and $\rho(0)=0$. Note that this implies that $\sgn(\psi(x))=\sgn(x)$.
\item \textbf{O3}: $\psi(|x|)=\gO(|x|^m)$ at infinity for some $m$. Furthermore, $\psi'$ is $L(u)$-Lipschitz on $(-|u|,|u|)$, where $L(|u|)\leq K |u|^{m_1}$ as $|u|\tendsto \infty$. Note that this implies that $\rho$ grows at most polynomially at $\infty$. 
\item \textbf{O4}: $X_i$'s are independent and identically distributed. Furthermore, for any 1-Lipschitz convex function $F$, $P(|F(X_i)-m_F|>t)\leq C_n \exp(-c_n t^2)$, $C_n$ and $c_n$ can vary with $n$. For simplicity, we assume that $c_n=\gO(1/(\log(n))^{\alpha})$ for some $\alpha\geq 0$.  $X_i$'s have mean 0 and $\scov{X_i}=\id_p$.
\item \textbf{O5}: $\{X_i\}_{i=1}^n$ are independent of $\{\eps_i\}_{i=1}^n$
\item \textbf{O6}: for any fixed $k$ , $\frac{1}{n}\sum_{i=1}^n \Exp{\psi^{2k}(\eps_i)}$ remains uniformly bounded in $p$ and $n$, as both grow to infinity.
\item \textbf{O7}: $\sup_{1\leq i \leq n} |\eps_i|\triangleq {\mathcal E}_n=\gO((\log n)^{\beta})$ and $\eps_i$'s are independent. 
\end{itemize}

For the second part of the proof (i.e ``leave-one-predictor-out''), we need all the previous assumptions and 
\begin{itemize}
\item \textbf{P1}: $X_i$'s have i.i.d entries. 
\end{itemize}
We note that according to Corollary 4.10 and the discussion that follows in \cite{ledoux2001}, Assumptions O4 and P1 are compatible. O4 is for instance satisfied
if the entries of $X_i$'s are bounded by $\gO((\log n)^{\alpha/2})$. Another example is the case of $X_i\sim {\cal N}(0,\id_p)$. 

For the last part of the proof, when we combine everything together, we will need the following assumptions on top of all the others:
\begin{itemize}
\item \textbf{F1}: the $\eps_i$'s have identical distribution and for any $r>0$, if $Z\sim{\cal N}(0,1)$, independent of $\eps_i$,  $\eps_i+rZ$ has a density $f$ which is increasing on $(-\infty,0)$ and decreasing on $(0,\infty)$. Furthermore, $\lim_{|t|\tendsto\infty} tf(t)=0$.
\item \textbf{F2}: For any fixed $k$, $\Exp{|\eps_i|^k}<\infty$. 
\end{itemize}

We refer the reader to Lemma \ref{lemma:suffCondKeyEqnInCHasUniqueSolution} and the discussion immediately following it for examples of such densities. We note that symmetric (around 0) $\log$-concave densities will for instance satisfy all the assumptions we made about the $\eps_i$'s. See \cite{KarlinTotalPositivity68} and \cite{IbragimovLogConcave1956} for instance. 

The aim of the paper is to prove the following theorem:
\begin{theorem}\label{thm:systemRidgeRegularizedRigorous}
Consider $\betaHat$ defined in Equation \eqref{eq:defBetaHatRidgeregularizedCase} and assume that $\tau>0$ is given. 
Under Assumptions \textbf{O1-O7}, \textbf{P1} and \textbf{F1-F2}, we have: as $p$, $n$ tend to infinity while $p/n\tendsto \kappa \in (0,\infty)$, $\var{\norm{\betaHat}}\tendsto 0$. Furthermore, if $\zHat_\eps=\eps+r_{\rho}(\kappa)Z$, where $\eps$ has the same distribution as $\eps_i$'s and $Z$ is a ${\cal N}(0,1)$ random variable independent of $\eps$, we have: $\norm{\betaHat}\tendsto r_{\rho}(\kappa)$ and there exists a constant $c_{\rho}(\kappa)$ such that 
\begin{equation}\label{eq:systemToProve}
\left\{
\begin{array}{cl}
\Exp{[\prox_{c_{\rho}(\kappa)}(\rho)]'(\zHat_\eps)}&=1-\kappa+\tau c_{\rho}(\kappa)\\
\kappa r^2_{\rho}(\kappa)&=\Exp{(\zHat_\eps-\prox_{c_{\rho}(\kappa)}(\rho)(\zHat_\eps))^2}\;.
\end{array}
\right.
\end{equation}
\end{theorem}

We use the notation $\prox_c(\rho)$ to denote the proximal mapping of the function $c\rho$. This notion was introduced in \cite{MoreauProxPaper65}. We recall that 
\begin{align*}
\prox_c(\rho)(x)&=\argmin_{y\in \mathbb{R}} (c\rho(y)+\frac{1}{2}(x-y)^2)\;, \text{ or equivalently,}\\
\prox_c(\rho)(x)&=(\id+c\psi)^{-1}(x)\;.
\end{align*}
The proximal mapping is an important notion in convex analysis and convex optimization (see for instance \cite{BeckAndTeboulleChapter2010} for a nice review of analytic properties and an introduction to proximal gradient algorithms). We note that even when $\rho$ is not differentiable, $\prox_c(\rho)(x)$ is a well-defined function.

As explained in \cite{NEKOptimalMEstimationPNASPublished2013}, the previous system can be reformulated in terms of $\prox_1((c_{\rho}(\kappa)\rho)^*)$, where $f^*$ represents the Fenchel-Legendre dual of $f$. 

\subsubsection*{Remarks on the assumptions}
In the context of robust statistics, where regression M-estimates are commonly used, $\rho$ is often taken to grow linearly at infinity. This is for instance the case for Huber functions. Furthermore, it will often be the case that for instance $\psi'$ is bounded. This situation arises if for instance $x\rightarrow x^2/2-\rho(x)$ is a convex function. So the growth conditions at infinity we impose on $\rho$ and $\psi$ are realistic for the problems we have in mind. A look at the proof reveals that if we had more restrictive growth conditions at infinity than the ones we impose, we could tolerate $\eps_i$'s with fewer moments and heavier tails. Understanding how heavy the tails of $\eps_i$ can be and the result still hold is interesting statistically, but we leave these considerations for future work. Conversely, our assumptions about $\eps_i$'s are somewhat restrictive - especially when it comes to their tail behavior. But this is just a consequence of our assumptions on $\rho$ and the fact that those are relatively unrestrictive. 

Assumption \textbf{O4} is a bit stronger than we will need. The functions $F$ we will be dealing with will either be linear or square-roots of quadratic forms. However, as documented in \cite{ledoux2001}, a large number of natural or ``reasonable" distributions satisfy the \textbf{O4} assumptions. Our choice of having a potentially varying $c_n$ is motivated by the idea that we could, for instance, relax an assumption of boundedness of the entries of $X_i$'s  - that guarantees that \textbf{O4} is satisfied when $X_i$ has i.i.d entries -  and replace it by an assumption concerning the moments of $X_i$'s: this is what we did for instance in \cite{nekCorrEllipD} through a truncation of triangular arrays argument. We also refer the interested reader to that paper for a short list of distributions satisfying \textbf{O4}. Finally, we could replace the $\exp(-c_n t^2)$ upper bound in $\textbf{O4}$ by $\exp(-c_n t^\alpha)$ for some fixed $\alpha>0$ and it seems that all our arguments would go through. We chose not to do work under these more general assumptions because it would involve extra book-keeping and does not enlarge the set of distributions we can consider enough to justify this extra technical cost. 

Our assumption that $1/c_n$ increases like a power of $\log(n)$ at most is quite restrictive when it comes to bounded random variables - but is of course satisfied by e.g Gaussian random variables where $c_n$ is a constant independent on $n$ - and motivated by simplifying the book-keeping needed in our proof. Having $1/c_n$ grow like $n^{\gamma}$ for a small $\gamma$ should be feasible - with $\gamma$ depending on $m$ and $m_1$. In the first part of the proof we keep track of the impact of $c_n$ to show this aspect of the problem. 

Statistically, regression $M$-estimates are quite widely used. But in the random design case studied here, they are known to have somewhat undesirable properties (\cite{baranchikInadmissibility73}, \cite{SteinInadmissibilityRegression60}) even in very simple situations. We do not dwell more on these otherwise interesting issues, since they are a bit tangential to the main aim of this particular paper, which is to give a rigorous justification of the heuristic manipulations made in  \cite{NEKRobustPaperPNAS2013Published}.

\subsubsection*{Notations}
We will repeatedly use the following notations: $\polyLog(n)$ is used to replace a power of $\log(n)$; $\lambda_{\max}(M)$ denotes the largest eigenvalue of the matrix $M$; $\opnorm{M}$ denotes the largest singular value of $M$. We call $\SigmaHat=\frac{1}{n}\sum_{i=1}^n X_i X_i\trsp$ the usual sample covariance matrix of the $X_i$'s. We say that $X\leq Y$ in $L_k$ if $\Exp{|X|^k}\leq \Exp{|Y|^k}$. We use the notation $u_n \lesssim v_n$ to say that there exists a constant $K$ independent of $n$ such that $u_n\leq K v_n$ for all $n$. We use the usual statistical notation $\betaHat_{(i)}$ to denote the regression vector we obtain when we do not use the pair $(X_i,Y_i)$ in our optimization problem. We will also use the notation $X_{(i)}$ to denote $\{X_1,\ldots,X_{i-1},X_{i+1},\ldots,X_n\}$. We use the notation $(a,b)$ for either the interval $(a,b)$ or the interval $(b,a)$: in several situations, we will have to localize quantities in intervals using two values $a$ and $b$ but we will not know whether $a<b$ or $b>a$. We denote by $X$ the $n\times p$ design matrix from $i$-th row is $X_i\trsp$. 

\subsubsection*{Remarks}

Note that under our assumptions on $\rho$, $\betaHat$ is defined as the solution of
\begin{align}
f(\betaHat)&=0 \textrm{ with }\\
f(\beta)&=\frac{1}{n}\sum_{i=1}^n -X_i \psi(\eps_i-X_i\trsp \beta)+\tau \beta \;.\label{eq:defGradient}
\end{align}
We call 
\begin{equation}\label{eq:defF}
F(\beta)=\frac{1}{n}\sum_{i=1}^n \rho(\eps_i-X_i\trsp \beta)+\frac{\tau}{2}\norm{\betaHat}^2\;.
\end{equation}
We call $R_i=\eps_i-X_i\trsp \betaHat$ (i.e the residuals), $S=\frac{1}{n}\sum_{i=1}^n \psi'(R_i) X_i X_i\trsp$ and $c_\tau=\frac{1}{n}\trace{S+\tau\id}^{-1}$.

\section{Preliminaries}

\subsection{General remarks}
\begin{proposition}\label{prop:ControlDeltaBetaDeltaf}
Let $\beta_1$ and $\beta_2$ be two vectors in $\mathbb{R}^p$. Then
\begin{equation}\label{eq:controlNormDeltaBetaFromNormDeltaF}
\boxed{
\norm{\beta_1-\beta_2}\leq\frac{1}{\tau}\norm{f(\beta_1)-f(\beta_2)}\;.
}
\end{equation}
When $\rho$ is strongly convex with modulus of convexity $C$, we also have
$$
\norm{\beta_1-\beta_2}\leq \frac{1}{C\lambda_{\min}(\SigmaHat)+\tau}\norm{f(\beta_1)-f(\beta_2)}\;.
$$	
\end{proposition}
\begin{proof}
Let $\beta_1$ and $\beta_2$ be two vectors in $\mathbb{R}^p$. We have
$$
f(\beta_1)-f(\beta_2)=\tau(\beta_1-\beta_2)+\frac{1}{n}\sum_{i=1}^n X_i \left[\psi(\eps_i-X_i\trsp\beta_2)-\psi(\eps_i-X_i\trsp\beta_1)\right]\;.
$$
We can use the mean value theorem to write
$$
\psi(\eps_i-X_i\trsp\beta_2)-\psi(\eps_i-X_i\trsp\beta_1)=\psi'(\gamma^*_{\eps_i,X_i\trsp\beta_1,X_i\trsp\beta_2}) X_i\trsp (\beta_1-\beta_2)\;,
$$
where $\gamma^*_{\eps_i,X_i\trsp\beta_1,X_i\trsp\beta_2}$ is in the interval $(\eps_i-X_i\trsp\beta_1,\eps_i-X_i\trsp\beta_2)$ - we do not care about the order of the endpoint in our notation.

We therefore have
$$
f(\beta_1)-f(\beta_2)=\tau(\beta_1-\beta_2)+\frac{1}{n}\sum_{i=1}^n \psi'(\gamma^*_{\eps_i,X_i\trsp\beta_1,X_i\trsp\beta_2})X_iX_i\trsp (\beta_1-\beta_2)\;,
$$
which we write
\begin{equation}\label{eq:exactRelationDeltafDeltaBeta}
f(\beta_1)-f(\beta_2)=(\mathsf{S}_{\beta_1,\beta_2}+\tau \id_p)(\beta_1-\beta_2)\;,	
\end{equation}
where
$$
\mathsf{S}_{\beta_1,\beta_2}=\frac{1}{n}\sum_{i=1}^n \psi'(\gamma^*_{\eps_i,X_i\trsp\beta_1,X_i\trsp\beta_2})X_iX_i\trsp\;.
$$
We therefore have
$$
\beta_1-\beta_2=(\mathsf{S}_{\beta_1,\beta_2}+\tau \id_p)^{-1}\left(f(\beta_1)-f(\beta_2)\right)\;.
$$

Since $\rho$ is convex, $\psi'=\rho''$ is non-negative and $\mathsf{S}_{\beta_1,\beta_2}$ is positive semi-definite. In the semi-definite order, we have
$\mathsf{S}_{\beta_1,\beta_2}+\tau \id \succeq \tau \id$. When $\rho$ is strongly convex with modulus C, we have $\psi'(x)\geq C$ (see Theorem 4.3.1 in \citet{HiriartLemarechalConvexAnalysisAbridged2001}) and therefore, $\mathsf{S}_{\beta_1,\beta_2}+\tau \id_p \succeq C \SigmaHat+\tau \id\succeq (C\lambda_{\min}(\SigmaHat)+\tau)\id_p$.
In particular,
$$
\norm{\beta_1-\beta_2}\leq\frac{1}{\tau}\norm{f(\beta_1)-f(\beta_2)}\;.
$$
In the strongly convex case, we have
$$
\norm{\beta_1-\beta_2}\leq\frac{1}{C\lambda_{\min}(\SigmaHat)+\tau}\norm{f(\beta_1)-f(\beta_2)}\;.
$$

\end{proof}
In the proof of Proposition \ref{prop:ControlDeltaBetaDeltaf}, it is clear that all we need is that ``enough" $\psi'(\gamma^*_{\eps_i,X_i\trsp\beta_1,X_i\trsp\beta_2})$'s are greater than a constant C. More precisely, let us call $N=\card{i: \psi'(\gamma^*_{\eps_i,X_i\trsp\beta_1,X_i\trsp\beta_2})\geq C}$ and let us call ${\cal I
}$ the corresponding set of indices. Results similar to that of Proposition \ref{prop:ControlDeltaBetaDeltaf} then hold, with $\SigmaHat$ being replaced by $\SigmaHat_{{\cal I}}=\frac{1}{n}\sum_{i\in {\cal I}} X_i X_i\trsp$. This could perhaps be used in certain situations to move away from strong convexity assumptions when we deal with the un-penalized (i.e $\tau=0$) case. See Section \ref{sec:regularizedToUnregularized} for more details about this question. Strong convexity is a very strong (and somewhat undesirable) requirement on $\rho$ for many applications in Statistics.

Proposition  \ref{prop:ControlDeltaBetaDeltaf} yields the following lemma.
\begin{lemma}\label{lemma:ControlApproxBetaHat}
For any $\beta_1$,
$$
\norm{\betaHat-\beta_1}\leq\frac{1}{\tau}\norm{f(\beta_1)}\;.
$$
\end{lemma}
The lemma is a simple consequence of Equation \eqref{eq:controlNormDeltaBetaFromNormDeltaF} since by definition $f(\betaHat)=0$\;.

In the following, we will strive to find approximations of $\betaHat$. We will therefore use Lemma \ref{lemma:ControlApproxBetaHat} repeatedly.

\subsection{Boundedness of $\norm{\betaHat}$}
We have the following lemma.
\begin{lemma}\label{lemma:betaHatIsBounded}
Let us call $W_n=\frac{1}{n}\sum_{i=1}^n X_i \psi(\eps_i)$.	
We have
$$
\norm{\betaHat}\leq \frac{1}{\tau}\norm{W_n}\;.
$$
In particular, when $X_i$ are independent and have covariance $\id_p$,
\begin{equation}\label{eq:boundNormBetaHatSquaredUsingPsiEpsi}
\Exp{\norm{\betaHat}^2}\leq \frac{1}{\tau^2}\frac{p}{n} \frac{1}{n}\sum_{i=1}^n \Exp{\psi^2(\eps_i)}\;.
\end{equation}

A similar result holds in $L_{2k}$ - provided the entries of $X_i$ has cumulants of order $2k$. This is automatically satisfied under our assumptions. 

This guarantees that $\norm{\betaHat}$ is bounded in $L_{2k}$ provided $\frac{1}{n}\sum_{i=1}^n\Exp{|\psi(\eps_i)|^{2k}}$ is bounded. If this latter quantity is $\polyLog(n)$ so is $\Exp{\norm{\betaHat}^{2k}}$.

We also have 
\begin{equation}\label{eq:boundNormBetaHatFromSimplyRidge}
\norm{\betaHat}\leq \sqrt{\frac{2}{\tau}} \sqrt{\frac{1}{n}\sum_{i=1}^n \rho(\eps_i)}\;,
\end{equation}
and hence 
$$
\Exp{\norm{\betaHat}^{2k}}\leq \frac{2^k}{\tau^k}\Exp{\left[\frac{1}{n}\sum_{i=1}^n \rho(\eps_i)\right]^k}\;.
$$

\end{lemma}
Though from a probabilistic point of view our various bounds might look interchangeable, it is important to have both from the point of view of statistical applications. Indeed, in robust regression, where $\eps_i$'s can have heavy tails, one would typically used bounded $\psi$ functions (for instance the Huber functions or smoothed version of the Huber functions - see \cite{HuberRonchettiRobustStatistics09}, p. 84, Equation (4.51) for a definition of the exponential of the Huber functions). The bound based on Equation 
\eqref{eq:boundNormBetaHatSquaredUsingPsiEpsi} is then particularly helpful. 
\newcommand{\OneVector}{\mathbf{e}}
\begin{proof}
The first inequality follows easily from taking $\beta_1=0$ in Lemma \ref{lemma:ControlApproxBetaHat}	and realizing that $W_n=f(0)$.
The second inequality follows from the fact that, if $\OneVector$ is an $n$-dimensional vector with entries all equal to 1,
$W_n=X\trsp D_\psi \OneVector/n$, where $X$ is $n\times p$ and $D_\psi$ is a diagonal matrix whose $(i,i)$ entry is $\psi(\eps_i)$. Hence,
$$
W_n^2=\frac{1}{n^2} \OneVector\trsp D_\psi XX\trsp D_\psi \OneVector\;,
$$
and therefore,
$\Exp{W_n^2}=\frac{p}{n^2}\sum_{i=1}^n \Exp{\psi^2(\eps_i)}$, since $\Exp{XX\trsp}=p\id_n$ and $\{\eps_i\}_{i=1}^n$ is independent of $\{X_i\}_{i=1}^n$.

For the $L_{2k}$ bound, can use $\Exp{\norm{W_n}^{2k}}\leq p^{k-1}\sum_{j=1}^p \Exp{W_n^{2k}(j)}$, because for $\alpha_i>0$, $(\sum_{i=1}^p \alpha_i)^k\leq p^{k-1}\sum \alpha_i^k$ by convexity.

Let us work temporarily conditional on $\eps_i$. 
We control  $\Exp{W_n^{2k}(i)}$ through the use of cumulants since $W_n(j)=\sum_{i=1}^n X_i(j)\psi(\eps_j)/n$, so the $2k$-th cumulant of $W_n(j)$ is $\sum_{i=1}^n \psi^{2k}(\eps_i)/n^{2k} \kappa_{2k}(X_i(j))$. These cumulants are all of order $n^{1-2k}$, if $\sum \psi^{2k}(\eps_i)/n=O(1)$. By the classical connection between moments and cumulants, we see that $\Exp{W_n^{2k}(j)}=\gO(n^{-k})$ if $\frac{1}{n}\sum_{i=1}^n \Exp{\psi^{2k}(\eps_i)}$. Hence,
$\Exp{\norm{W_n}^{2k}}=\gO(p^{k-1}pn^{-k})=\gO(1)$.

The proof of Equation \eqref{eq:boundNormBetaHatFromSimplyRidge} simply follows from observing that 
\begin{align*}
\frac{\tau}{2}\norm{\betaHat}^2&\leq \frac{1}{n}\sum_{i=1}^n \rho(\eps_i-X_i\trsp\betaHat)+\frac{\tau}{2}\norm{\betaHat}^2\\
&\leq \frac{1}{n}\sum_{i=1}^n \rho(\eps_i)\;.
\end{align*}
Indeed, since, according to Equation \eqref{eq:defF},
$$
\betaHat=\argmin_{\beta} F(\beta)\;,
$$
we also have 
$$
F(\betaHat)\leq F(0)=\frac{1}{n}\sum_{i=1}^n \rho(\eps_i)\;,
$$
and the result follows immediately.

\end{proof}

\section{Approximating $\betaHat$ by $\betaHat_{(i)}$: leave-one-observation-out}
We consider the situation where we leave one observation out. We call
\begin{align*}
\rTilde_{j,(i)}&=\eps_j-X_j\trsp \betaHat_{(i)} \text{ and }\\
S_i&=\frac{1}{n}\sum_{j\neq i}\psi'(\rTilde_{j,(i)})X_jX_j\trsp\;.
\end{align*}

We also call
$$
f_i(\beta)=-\frac{1}{n}\sum_{j\neq i}X_j \psi(\eps_j-X_j\trsp \beta)+\tau \beta\;.
$$

We call $\betaHat_{(i)}$ the solution of $f_i(\betaHat_{(i)})=0$ and call it the leave-one-out estimate.

Let us consider
\begin{equation}\label{eq:ApproxBetaHatStandardLeaveOneOut}
\betaTilde_i=\betaHat_{(i)}+\frac{1}{n}(S_i+\tau \id)^{-1}X_i \psi(\prox_{c_i}(\rho)(\rTilde_{i,(i)}))\triangleq \betaHat_{(i)}+\eta_i\;,
\end{equation}
where 
\begin{align}\label{eq:defciandetaifirstpart}
c_i&=\frac{1}{n}X_i\trsp (S_i+\tau \id)^{-1}X_i\;, \text{ and }\\
\eta_i&=\frac{1}{n}(S_i+\tau \id)^{-1}X_i \psi(\prox_{c_i}(\rho)(\rTilde_{i,(i)}))\;.
\end{align}

All these approximations are ``very natural'' in light of the probabilistic heuristics we derived for this problem in \cite{NEKRobustPaperPNAS2013Published} - so we refer the reader to that paper for explanations about why we choose to introduce these quantities. One of the aim of the paper is to show that these heuristics are valid and indeed open up the horizon to rigorous proofs. 

The aim of the work that follows is to show that $\betaHat$ can be very well approximated by $\betaTilde_i$. In Corollary \ref{coro:AggregResApproxBetaHatByBetaTildeIncl}, we show that the approximation is accurate to order $\polyLog(n)/n$ in Euclidian norm, if for instance $1/c_n=\polyLog(n)$. We refer the reader to Corollary \ref{coro:AggregResApproxBetaHatByBetaTildeIncl} for full details. 

\subsection{Deterministic bounds}

\begin{proposition}\label{prop:ControlDeltaBetaHatBetaTildei}
We have
\begin{equation}\label{eq:approxBetaHatbyBetaTildei}
\norm{\betaHat-\betaTilde_i}\leq \frac{1}{\tau} \norm{{\mathcal R}_i}\;,
\end{equation}
where
\begin{equation}\label{eq:definitionScriptRiLOOO}
{\mathcal R}_i=\frac{1}{n}\sum_{j\neq i} \left[\psi'(\gamma^*(X_j,\betaHat_{(i)},\eta_i))-\psi'(\rTilde_{j,(i)})\right] X_jX_j\trsp \eta_i\;,
\end{equation}
and $\gamma^*(X_j,\betaHat_{(i)},\eta_i)$ is in the (``unordered") interval $(\rTilde_{j,(i)},\rTilde_{j,(i)}-X_j\trsp\eta_i)$.
\end{proposition}
\begin{proof}
We have of course,
$$
f(\betaTilde_i)=f(\betaTilde_i)-f_i(\betaHat_{(i)})=-\frac{1}{n}X_i \psi(\eps_i-X_i\trsp\betaTilde_i)+\frac{1}{n}\sum_{j\neq i}X_j\left[\psi(\eps_j-X_j\trsp \betaHat_{(i)})-\psi(\eps_j-X_j\trsp (\betaHat_{(i)}+\eta_i))\right]+\tau \eta_i\;.
$$

By the mean-value theorem, we also have
$$
\psi(\eps_j-X_j\trsp \betaHat_{(i)})-\psi(\eps_j-X_j\trsp (\betaHat_{(i)}+\eta_i))=\psi'(\rTilde_{j,(i)})X_j\trsp \eta_i + \left[\psi'(\gamma^*(X_j,\betaHat_{(i)},\eta_i))-\psi'(\rTilde_{j,(i)})\right] X_j\trsp \eta_i\;,
$$
where $\gamma^*(X_j,\betaHat_{(i)},\eta_i)$ is in the (``unordered") interval $(\eps_j-X_j\trsp \betaHat_{(i)},\eps_j-X_j\trsp (\betaHat_{(i)}+\eta_i))$, i.e
$(\rTilde_{j,(i)},\rTilde_{j,(i)}-X_j\trsp\eta_i)$.

Hence, if ${\mathcal R}_i$ is the quantity defined in Equation \eqref{eq:definitionScriptRiLOOO},
\begin{align*}
\frac{1}{n}\sum_{j\neq i}X_j\left[\psi(\eps_j-X_j\trsp \betaHat_{(i)})-\psi(\eps_j-X_j\trsp (\betaHat_{(i)}+\eta_i))\right]&=\frac{1}{n}\sum_{j\neq i}\psi'(\rTilde_{j,(i)})X_jX_j\trsp \eta_i + {\mathcal R}_i\;,\\
&=S_i \eta_i+{\mathcal R}_i\;.
\end{align*}

In light of the previous simplifications, we have
$$
f(\betaTilde_i)=-\frac{1}{n}X_i \psi(\eps_i-X_i\trsp\betaTilde_i)+(S_i+\tau \id) \eta_i +{\mathcal R}_i\;.
$$
Since by definition, $\eta_i=\frac{1}{n}(S_i+\tau \id)^{-1}X_i \psi(\prox_{c_i}(\rho)(\rTilde_{i,(i)}))$,
$$
(S_i+\tau \id) \eta_i=\frac{1}{n}X_i \psi(\prox_{c_i}(\rho)(\rTilde_{i,(i)}))\;.
$$
In other respects,
$$
\eps_i-X_i\trsp\betaTilde_i=\rTilde_{i,(i)}-c_i \psi(\prox_{c_i}(\rho)(\rTilde_{i,(i)}))\;.
$$
When $\psi$ is differentiable, $x-c\psi(\prox_c(\rho)(x))=\prox_c(\rho)(x)$ almost by definition of the proximal mapping (see Lemma \ref{lemma:ValueProxat0} and its proof). Therefore,
$$
-\frac{1}{n}X_i \psi(\eps_i-X_i\trsp\betaTilde_i)+(S_i+\tau \id) \eta_i=\frac{1}{n}X_i \left[-\psi(\prox_{c_i}(\rho)(\rTilde_{i,(i)}))+\psi(\prox_{c_i}(\rho)(\rTilde_{i,(i)}))\right]=0.
$$
We conclude that
$$
f(\betaTilde_i)={\mathcal R}_i\;.
$$
Applying Lemma \ref{lemma:ControlApproxBetaHat}, we see that
$$
\norm{\betaHat-\betaTilde_i}\leq\frac{1}{\tau}\norm{{\mathcal R}_i}\;.
$$
\end{proof}
\subsubsection{On $\bm{{\mathcal R}_i}$}

\begin{lemma}\label{lemma:controlNormRemainderFLeaveOneOutAdjusted}
We have
\begin{equation}\label{eq:ControlEtai}
\norm{\eta_i}\leq \frac{1}{\sqrt{n}\tau}\frac{\norm{X_i}}{\sqrt{n}} \left[|\psi(\rTilde_{i,(i)})|\wedge \frac{|\rTilde_{i,(i)}|}{c_i}\right]\;,
\end{equation}	
and
\begin{equation}\label{eq:ControlNormRemainderFLeaveOneOutAdjusted}
\norm{{\mathcal R}_i}\leq \opnorm{\SigmaHat} \sup_{j\neq i}\left|\psi'(\gamma^*(X_j,\betaHat_{(i)},\eta_i))-\psi'(\rTilde_{j,(i)})\right|
\frac{1}{\sqrt{n}\tau}\frac{\norm{X_i}}{\sqrt{n}} \left[|\psi(\rTilde_{i,(i)})|\wedge |\rTilde_{i,(i)}|/c_i\right]\;.
\end{equation}
\end{lemma}

\begin{proof}
We have
$$
{\mathcal R}_i=\frac{1}{n}\sum_{j\neq i} \left[\psi'(\gamma^*(X_j,\betaHat_{(i)},\eta_i))-\psi'(\rTilde_{j,(i)})\right] X_jX_j\trsp \eta_i\;.
$$
Of course, ${\mathcal S}=\frac{1}{n}\sum_{j\neq i} \left[\psi'(\gamma^*(X_j,\betaHat_{(i)},\eta_i))-\psi'(\rTilde_{j,(i)})\right] X_jX_j\trsp$ can be written
${\mathcal S}=\frac{1}{n}X\trsp {\mathsf D} X$, where ${\mathsf D}$ is a diagonal matrix with $(j,j)$ entry $\left[\psi'(\gamma^*(X_j,\betaHat_{(i)},\eta_i))-\psi'(\rTilde_{j,(i)})\right]$ and $(i,i)$ entry 0.

Using the fact that $\opnorm{\cdot}$ is a matrix norm, we see that $\opnorm{{\mathcal S}}\leq \opnorm{\SigmaHat}\opnorm{{\mathsf D}}$. This implies that
$$
\norm{{\mathcal R}_i}\leq \opnorm{\SigmaHat} \sup_{j\neq i}\left|\psi'(\gamma^*(X_j,\betaHat_{(i)},\eta_i))-\psi'(\rTilde_{j,(i)})\right| \norm{\eta_i}\;,
$$
where $\SigmaHat=\frac{1}{n}\sum_{i=1}^nX_i X_i\trsp$ is the usual sample covariance matrix.

We note that
$$
\norm{\eta_i}\leq \frac{1}{\sqrt{n}\tau}\frac{\norm{X_i}}{\sqrt{n}} |\psi(\prox_{c_i}(\rho)(\rTilde_{i,(i)}))|\;.
$$

Using Lemma \ref{lemma:ValueProxat0}, we see that
$$
|\psi(\prox_{c_i}(\rho)(\rTilde_{i,(i)}))|\leq |\psi(\rTilde_{i,(i)})|\wedge \frac{|\rTilde_{i,(i)}|}{c_i}\;.
$$
The lemma is shown.
\end{proof}

\subsubsection{On $\gamma^*(X_j,\betaHat_{(i)},\eta_i)$ and related quantities}
We now show how to control $\frac{1}{\sqrt{n}}\sup_{j\neq i}\left|\psi'(\gamma^*(X_j,\betaHat_{(i)},\eta_i))-\psi'(\rTilde_{j,(i)})\right|$
\begin{lemma}
Let us call 
$$
{\cal B}_n(i)=\sup_{j\neq i}\left[|\eps_j-X_j\trsp \betaHat_{(i)}|+|X_j\trsp \eta_i|\right]\;.
$$
Suppose, as in our assumptions, that $\psi'$ is $L({\cal B}_n(i))$ Lipschitz on $(-{\cal B}_n(i),{\cal B}_n(i))$. Then, 
$$
\sup_{j\neq i}\left|\psi'(\gamma^*(X_j,\betaHat_{(i)},\eta_i))-\psi'(\rTilde_{j,(i)})\right|\leq L({\cal B}_n(i)) \sup_{j\neq i}|X_j\trsp \eta_i|\;.
$$
It follows that
$$
\norm{{\mathcal R}_i}\leq
  \sup_{j\neq i}|X_j\trsp \eta_i|
\frac{L({\cal B}_n(i))}{\sqrt{n}\tau}\frac{\norm{X_i}}{\sqrt{n}} \opnorm{\SigmaHat} \left[|\psi(\rTilde_{i,(i)})|\wedge |\rTilde_{i,(i)}|/c_i\right]\;.
$$
\end{lemma}
We note that we could replace the assumption concerning the Lipschitz property of $\psi'$ on $(-{\cal B}_n(i),{\cal B}_n(i))$ by saying that $\psi'$ has modulus of continuity $\omega_n$ when restricted to this interval and putting growth condition on this modulus. We chose not to do this to simplify the exposition. 
\begin{proof}
By definition, we have
$$
|\gamma^*(X_j,\betaHat_{(i)},\eta_i)-\rTilde_{j,(i)}|\leq |X_j\trsp \eta_i|\;.
$$
Of course,
$$
X_j\trsp \eta_i=\psi(\prox_{c_i}(\rho)(\rTilde_{i,(i)})) \frac{1}{n}X_j\trsp (S_i+\tau \id_p)^{-1}X_i\;.
$$

Therefore,
$$
\sup_{j\neq i}|\gamma^*(X_j,\betaHat_{(i)},\eta_i)|\leq \sup_{j\neq i}\left[|\eps_j-X_j\trsp \betaHat_{(i)}|+|X_j\trsp \eta_i|\right]
$$
We call
$$
{\cal B}_n(i)=\sup_{j\neq i}\left[|\eps_j-X_j\trsp \betaHat_{(i)}|+|X_j\trsp \eta_i|\right]\;.
$$

Since, by assumption, $\psi'$ is $L({\cal B}_n(i))$-Lipschitz on $(-{\cal B}_n(i),{\cal B}_n(i))$. Then
$$
\sup_{j\neq i}\left|\psi'(\gamma^*(X_j,\betaHat_{(i)},\eta_i))-\psi'(\rTilde_{j,(i)})\right|\leq L({\cal B}_n(i)) \sup_{j\neq i}|X_j\trsp \eta_i|\;.
$$

The bound for $\norm{{\mathcal R}_i}$ follows immediately.

\end{proof}

\subsection{Probabilistic aspects}
We can rewrite the bound on $\norm{{\mathcal R}_i}$ as 
$$
\norm{{\mathcal R}_i}\leq
  \left[\sup_{j\neq i}\frac{|X_j\trsp (S_i+\tau \id_p)^{-1}X_i|}{n}\right]
\frac{L({\cal B}_n(i))}{\sqrt{n}\tau}\frac{\norm{X_i}}{\sqrt{n}} \opnorm{\SigmaHat} \left(\left[|\psi(\rTilde_{i,(i)})|\wedge |\rTilde_{i,(i)}|/c_i\right] |\psi(\prox_{c_i}(\rho)(\rTilde_{i,(i)})|\right)\;.
$$

The bound on $\norm{{\mathcal R}_i}$ is encouraging since it shows that we can control $\norm{\betaHat-\betaTilde_i}$ in $L_k$ provided we can control each terms in the product in $L_{5k}$: indeed, for a product of $k$ random variables $\{W_j\}_{j=1}^k$, we have $\Exp{|\prod_{j=1}^k W_j|}\leq \prod_{j=1}^k \left[\Exp{|W_j|^k}\right]^{1/k}$ by H\"older's inequality. In particular, we will later need control of $\Exp{\norm{\betaHat-\betaTilde_i}^2}$ and will therefore require subsequent bounds to in $L_{10}$. 

\subsubsection{On $\sup_{j\neq i}|X_j \trsp (S_i+\tau \id)^{-1} X_i/n|$}
We will control $X_j\trsp (S_i+\tau \id)^{-1} X_i/n$ by appealing to Lemma \ref{lemma:controlRandomLipFuncConcRVsLp}.
\begin{lemma}\label{lemma:keyLemmaControlXjtransposeetai}
Suppose $X_i$ are independent and satisfy the concentration assumptions mentioned above. 
Then 
$$
\sup_{j\neq i} |X_j\trsp (S_i+\tau \id)^{-1} X_i/n|\leq \frac{1}{\sqrt{n}} \sup_{j \neq i} \frac{\norm{X_j}}{\tau\sqrt{n}} \polyLog(n)/c^{1/2}_n
$$
in $L_{10}$, provided we control $\sup_{j \neq i} \frac{\norm{X_j}}{\tau\sqrt{n}}$ in $\sqrt{L_{20}}$. 
\end{lemma}
We use the perhaps slightly unusual notation $\sqrt{L_{20}}$ to simply say that we control $\sqrt{\Exp{Z^{20}}}$ for a random variable $Z$. 
\begin{proof}
Let us work conditionally on $X_{(i)}=\{X_1,\ldots,X_{i-1},X_{i+1},\ldots,X_n\}$. Call $v_{j,(i)}= (S_i+\tau \id)^{-1} X_j$. The map $F_j(X_i)=X_j\trsp (S_i+\tau \id)^{-1} X_i=X_i\trsp v_{j,(i)}$ is Lipschitz (as a function of $X_i$) with Lipschitz constant 
$\sqrt{X_j\trsp (S_i+\tau \id)^{-2}X_j}\leq \norm{X_j}/\tau$. Indeed, it is linear in $X_i$.
Therefore, using Lemma \ref{lemma:controlRandomLipFuncConcRVsLp}, we see that 
$$
\frac{1}{n}\sup_{j\neq i} |X_j\trsp (S_i+\tau \id)^{-1} X_i||X_{(i)}\leq \frac{1}{\sqrt{n}} \sup_j \frac{\norm{X_j}}{\tau\sqrt{n}} \sqrt{\polyLog n/c_n}+\sup_j |m_{F_j}|\;. 
$$
with overwhelming ($X_i$)-probability and in $L_{10}$. Recall that in Lemma \ref{lemma:controlRandomLipFuncConcRVsLp}, we have a choice between the mean and the median for the definition of $m_{F_j}$. Here we choose the mean. 

Since $X_i$ has mean 0, we see that $m_{F_j}=0$, so that 
$$
\frac{1}{n}\sup_{j\neq i} |X_j\trsp v_{j,(i)}| |X_{(i)}\leq \frac{1}{\sqrt{n}} \sup_{j \neq i} \frac{\norm{X_j}}{\tau\sqrt{n}} \sqrt{\polyLog n/c_n}
$$
with overwhelming ($X_i$)-probability and in $L_{10}$. We can then integrate over $X_{(i)}$ to get the result. 
\end{proof}

We note that using the fact that $X_j\rightarrow \norm{X_j}/\sqrt{n}$ is $n^{-1/2}$-Lipschitz we see that 
$$
\sup_{j\neq i}|\norm{X_j}/\sqrt{n}-m_{\norm{X_j}/\sqrt{n}}|\leq \polyLog(n)/(\sqrt{nc_n}) \text{ in } \sqrt{L_{20}}\;.
$$
Recall that $\scov{X_i}=\id_p$. So $m_{\norm{X_j}/\sqrt{n}}$ is of order 1 in the case we are interested in, we see that 
$$
\sup_{j\neq i}|\norm{X_j}/\sqrt{n}|=\gO_{\sqrt{L_{20}}}(1)\;,
$$
provided $nc_n \gg \polyLog(n)$. This is clearly the case under our assumptions.

\subsubsection{Control of the residuals $R_i$ and $\rTilde_{i,(i)}$}
Our aim here is to show that we can control $\sup_i |R_i|$, where $R_i=\eps_i-X_i\trsp \betaHat$ are the residuals from the full ridge-regression model. This will allow us to achieve control of ${\mathcal B}_n(i)$. As $\rTilde_{i,(i)}$ is much easier to understand than $R_i$, our strategy is to relate the two. 
\begin{lemma}\label{lemma:BoundRiFromrTildeii}
We have the deterministic bound
\begin{equation}\label{eq:ControlRi}
|R_i|\leq |\rTilde_{i,(i)}|+\frac{\norm{X_i}^2}{n}\frac{1}{\tau} |\psi(\rTilde_{i,(i)})|\;.
\end{equation}
Denoting by ${\mathcal E}_n=\sup_{1\leq i \leq n} |\eps_i|$, we have under our assumptions on $\{X_i\}_{i=1}^n$, 
$$
\sup_{1\leq i \leq n}|\rTilde_{i,(i)}|\leq  {\mathcal E}_n+[\norm{W_{n}}+\frac{1}{n}\sup_{1\leq i\leq n} \norm{X_i}|\psi({\mathcal E}_n)\vee \psi(-{\mathcal E}_n)|] \polyLog(n)/\sqrt{c_n}\text{ in }L_k\;.
$$
Under the assumption that $|\psi(x)|=\gO(|x|^m)$ for some fixed $m$ at infinity, we have 
$$
\sup_i |R_i|\leq K (\sup_i |\rTilde_{i,(i)}|)^{m\vee 1}) \text{ in } L_k\;,
$$
and 
$\norm{W_{n}}+\frac{1}{n}\sup_{1\leq i\leq n} \norm{X_i}|\psi({\mathcal E}_n)\vee \psi(-{\mathcal E}_n)|=\gO_{L_k}(\norm{W_{n}}+{\mathcal E}_n^m/\sqrt{n})$.
\end{lemma}

\begin{proof}
Recall the representation
$$
\beta_1-\beta_2=(\mathsf{S}_{\beta_1,\beta_2}+\tau \id_p)^{-1}\left(f(\beta_1)-f(\beta_2)\right)\;.
$$
Take $\beta_1=\betaHat$ and $\beta_2=\betaHat_{(i)}$.
Note that
$$
f(\betaHat_{(i)})=-\frac{1}{n}X_i \psi(\rTilde_{i,(i)})-\frac{1}{n}\sum_{j\neq i} X_i \psi(\rTilde_{j,(i)})+\tau \betaHat_{(i)}=-\frac{1}{n}X_i \psi(\rTilde_{i,(i)})
$$
by definition of $\betaHat_{(i)}$. Therefore,
$$
\betaHat-\betaHat_{(i)}=\frac{1}{n}(\mathsf{S}_{\betaHat,\betaHat_{(i)}}+\tau \id_p)^{-1}X_i \psi(\rTilde_{i,(i)})\;.
$$
Since $\rTilde_{i,(i)}-R_i=X_i\trsp (\betaHat-\betaHat_{(i)})$, we also have
$$
|\rTilde_{i,(i)}-R_i|\leq \frac{\norm{X_i}^2}{n}\frac{1}{\tau} |\psi(\rTilde_{i,(i)})|\;.
$$
We conclude that
$$
|R_i|\leq |\rTilde_{i,(i)}|+\frac{\norm{X_i}^2}{n}\frac{1}{\tau} |\psi(\rTilde_{i,(i)})|\;.
$$
Now under assumptions, we have $\sup_{1\leq i\leq n}|\norm{X_i}^2/n-1|=\gO_{L_k}(\polyLog(n)/\sqrt{nc_n})$, according to either Lemma \ref{lemma:controlRandomLipFuncConcRVsLp} or Lemma \ref{lemma:RandomQuadFormsBounds}.
Using the fact that $\norm{\betaHat_{(i)}}\leq \norm{W_{n,(i)}}$ (see Lemma \ref{lemma:betaHatIsBounded}), the independence of $X_i$ and $\betaHat_{(i)}$, we have,
through Lemma \ref{lemma:controlRandomLipFuncConcRVsLp}, 
$$
\sup_{1\leq i \leq n} |X_i\trsp \betaHat_{(i)}|\leq \sup_{1\leq i\leq n} \norm{W_{n,(i)}} \polyLog(n)/\sqrt{c_n}\;.
$$
Since $\norm{W_{n,(i)}} \leq \norm{W_{n}}+\norm{X_i}|\psi(\eps_i)|/n$, we have

$$
|\rTilde_{i,(i)}|\leq |\eps_i|+[\norm{W_{n}}+\sup_i \norm{X_i}|\psi(\eps_i)|/n)] \polyLog(n)/\sqrt{c_n} \text{in } L_k\;.
$$
Denoting by ${\mathcal E}_n=\sup_{1\leq i \leq n} |\eps_i|$, we have, using the fact that $\psi$ is increasing, 
$$
\sup_{1\leq i \leq n}|\rTilde_{i,(i)}|\leq  {\mathcal E}_n+\left[\norm{W_{n}}+\frac{1}{n^{1/2}}\sup_{1\leq i\leq n} \frac{\norm{X_i}}{n^{1/2}}|\psi({\mathcal E}_n)\vee \psi(-{\mathcal E}_n)|\right] \polyLog(n)/\sqrt{c_n}\text{ in }L_k\;,
$$
for any given $k$. 
We note that if $|\psi(x)|=\gO(x^m)$ at $\infty$, we have the bound $\sup_{1\leq i\leq n}|R_i| \lesssim \sup_{1\leq i\leq n}|\rTilde_{i,(i)}|^{m\vee 1}$ and therefore,

$$
\sup_{1\leq i\leq n}|R_i|\lesssim \left[{\mathcal E}_n+\polyLog(n)/\sqrt{c_n}[\norm{W_{n}}+\frac{1}{n}\sup_{1\leq i\leq n} \norm{X_i}|\psi({\mathcal E}_n)\vee \psi(-{\mathcal E}_n)|\right]^{m \vee 1} \text{ in }L_k\;,
$$
provided the bound on $\sup_{1\leq i \leq n}|\rTilde_{i,(i)}|$ holds in $L_{mk}$. Note that this is guaranteed under our assumptions. Of course, here we are using control of $\sup_i \norm{X_i}^2/n$, which we get by controlling $\norm{X_i}/\sqrt{n}$ through concentration arguments.
The fact that $\sup_i \norm{X_i}/\sqrt{n}=\gO_{L_k}(1)$ gives us the last statement of the lemma.
\end{proof}
\textbf{Remark 1: } at the gist of the bound on $\rTilde_{i,(i)}$ is a uniform bound on $\norm{\betaHat_{(i)}}$ in $L_k$. If one is not concerned about having assumptions that limit the existence of moments for $\sqrt{1/n\sum_{i=1}^n \rho(\eps_i)}$, one could use the bound $\sup_i \norm{\betaHat_{(i)}}\leq \sqrt{2/\tau}\sqrt{1/n\sum_{i=1}^n \rho(\eps_i)}$ which is immediate from Lemma \ref{lemma:betaHatIsBounded}. This would change slightly the appearance of our bounds on $\sup_i |\rTilde_{i,(i)}|$. In particular, under our assumptions, this bound is valid.\\
\textbf{Remark 2: } We note that a similar result holds of course for $\rTilde_{j,(i)}$. More precisely,
$$
|\rTilde_{j,(i)}-R_j|\leq \left|\frac{1}{n}X_j\trsp(\mathsf{S}_{\betaHat,\betaHat_1}+\tau \id_p)^{-1}X_i\right| \left|\psi(\rTilde_{i,(i)})\right|\;,
$$
and hence,
$$
|\rTilde_{j,(i)}-R_j|\leq \frac{\norm{X_j}\norm{X_i}}{n\tau} \left|\psi(\rTilde_{i,(i)})\right|\;.
$$
Of course, this bound is very coarse and we will see that we can get a better one later.

However, this finally allows us to have the following proposition
\begin{proposition}\label{prop:ControlBni}
Under the assumption that $|\psi(x)|=\gO(|x|^m)$, we have the bound 	
$$
{\cal B}_n(i)\leq K\left[ {\mathcal E}_n+(\norm{W_{n}}+\frac{{\mathcal E}_n^m}{\sqrt{n}}) \polyLog(n)/\sqrt{c_n}\right]^{m\vee 1}\text{ in }L_k\;,
$$
where $K$ is a constant independent of $p$ and $n$. When $\norm{W_{n}}$ and $\frac{{\mathcal E}_n^m}{\sqrt{n}}$ are bounded in $L_k$, this bound simply becomes 
$$
{\cal B}_n(i)\leq K \left[{\mathcal E}_n\vee \polyLog(n)/\sqrt{c_n}\right]^{m\vee 1}\text{ in }L_k\;.
$$
The same bound holds for $\sup_{i}{\cal B}_n(i)$ in $L_k$\;.
\end{proposition}

\begin{proof}
The result follows easily from the fact that 
$$
{\cal B}_n(i)=\sup_{j\neq i}\left[|\rTilde_{j,(i)}|+|X_j\trsp \eta_i|\right]\;,
$$
the fact that 
$$
\sup_i \sup_{j\neq i}|\rTilde_{j,(i)}-R_j|\leq \sup_i \sup_j \frac{\norm{X_j}\norm{X_i}}{n\tau} \left|\psi(\rTilde_{i,(i)})\right|\;,
$$		
and the bounds on $\sup_i|R_i|$ we have derived earlier. The part concerning $\sup_i \sup_{j\neq i}|X_j\trsp \eta_i|$ is easily shown to be negligible compared to this quantity from our previous investigations concerning $X_j\trsp (S_i+\tau \id_p)^{-1}X_i$. 
\end{proof}

\subsubsection{Consequences}
We have the following result. Recall that $\psi'$ is assumed to be Lipschitz with Lipschitz constant $L(u)$ on $(-|u|,|u|)$.
\begin{proposition}\label{prop:controlNormRiStochUnderAssumptions}
Suppose that $|\psi(x)|=\gO(|x|^m)$ and ${\cal E}^m_n=\lo(\sqrt{n})$ in $L_k$. Suppose further that $L(x)\leq K |x|^{m_1}$. Then we have 
$$
\norm{{\cal R}_i}\leq K \frac{\polyLog(n)}{nc_n}\left({\mathcal E}_n\vee (c_n)^{-1/2}\polyLog(n)\right)^{2m+m_1} \text{ in } L_k\;.
$$	
In particular, if ${\mathcal E}_n=\polyLog(n)$ and $1/c_n=\gO(\polyLog(n))$, we have 
$$
\norm{{\cal R}_i}\leq K \frac{\polyLog(n)}{n} \text{ in } L_k\;.
$$
Furthermore, the same bounds hold for $\sup_i \norm{{\cal R}_i}$. 
\end{proposition}

\begin{proof}
The proof follows by aggregating all the intermediate results we had and noticing that under our assumptions, $\opnorm{\SigmaHat}=\gO_{L_k}(c_n^{-1/2})$. This latter result follows easily from a standard $\eps$-net and union bound argument for controlling $\opnorm{\SigmaHat}$ - see e.g \cite{TalagrandSpinGlassesBook03}, Appendix A.4. We provide some details in Lemma \ref{lemma:controlSpectralNorms}. 

The statement concerning $\sup_i \norm{{\cal R}_i}$ follows by the same method. 
\end{proof}

We have the following very important corollary.
\begin{corollary}\label{coro:AggregResApproxBetaHatByBetaTildeIncl}
Under Assumptions \textbf{O1}-\textbf{O7}, we have 
$$
\norm{\betaHat-\betaTilde_i}=\gO_{L_k}(\frac{\polyLog(n)}{n})
$$
In particular, we have 
$$
\Exp{\norm{\betaHat-\betaTilde_i}^2}=\gO(\polyLog(n)/n^2)\;.
$$

Also, 
$$
\sup_{1\leq i\leq n} \sup_{ j\neq i}|\rTilde_{j,(i)}-R_j|=\gO_{L_k}(\frac{\polyLog(n)}{n^{1/2}})\;.
$$

Finally, 
$$
\sup_i |R_i-\prox_{c_i}(\rho)(\rTilde_{i,(i)})|=\gO_{L_k}(\frac{\polyLog(n)}{n^{1/2}})\;.
$$

\end{corollary}

The only parts that may require a discussion are the ones involving the residuals. However, they follow easily from the very coarse bound 
\begin{align*}
\sup_{j\neq i}|\rTilde_{j,(i)}-R_j|&=\sup_{ j\neq i}\left|X_j\trsp (\betaHat-\betaHat_i)\right|\leq \sup_{ j\neq i}\left|X_j\trsp (\betaHat-\betaTilde_i)\right|+\sup_{j\neq i} |X_j\trsp (\betaTilde_i-\betaHat_i)|\;,\\
&\leq \left(\sup_{1\leq j\leq n}\frac{\norm{X_j}}{\sqrt{n}}\right) \sqrt{n}\norm{\betaHat-\betaTilde_i}+\sup_{j\neq i} |X_j\trsp \eta_i|\;,
\end{align*}
and the fact that $\left(\sup_{1\leq j\leq n}\frac{\norm{X_j}}{\sqrt{n}}\right)=\gO_{L_k}(1)$ under our assumptions.
Recalling that $\norm{\betaHat-\betaTilde_i}\leq \norm{{\mathcal R}_i}$ and hence $\sup_i \norm{\betaHat-\betaTilde_i}\leq \sup_i \norm{{\mathcal R}_i}$ gives control of the first term. Control of the second term follows basically from Lemma \ref{lemma:keyLemmaControlXjtransposeetai}. 

Concerning the approximation of $R_i$, recall that 
$$
R_i=\eps_i-X_i\trsp\betaHat=\eps_i-X_i\trsp \betaTilde_i-X_i\trsp (\betaHat-\betaTilde_i)\;.
$$
Now, given the definition of $\betaTilde_i$, we have 
$$
X_i\trsp \betaTilde_i=X_i\trsp\betaHat_i+c_i \prox_{c_i}(\rho)(\rTilde_{i,(i)})\;.
$$
Hence, 
$$
\eps_i-X_i\trsp \betaTilde_i=\rTilde_{i,(i)}-c_i\prox_{c_i}(\rho)(\rTilde_{i,(i)})=\prox_{c_i}(\rho)(\rTilde_{i,(i)})\;,
$$
where the last equality is a standard property of the proximal mapping (see Lemma \ref{lemma:ValueProxat0} if needed).
So we have established that 
$$
\sup_{i}\left|R_i-\prox_{c_i}(\rho)(\rTilde_{i,(i)})\right|=\sup_i \left|X_i\trsp (\betaTilde_i-\betaHat)\right|
$$
and the result follows from our previous bounds.

\subsection{Asymptotically deterministic character of $\norm{\betaHat}^2$}
\begin{proposition}
Under our assumptions, 
$$
\var{\norm{\betaHat}^2}\tendsto 0 \text{ as } n\tendsto \infty\;.
$$
Therefore $\norm{\betaHat}^2$ has a deterministic equivalent in probability and in $L_2$.	

In particular, when $c_n=1/\polyLog(n)$, we have 
$$
\var{\norm{\betaHat}^2}=\gO(\frac{\polyLog(n)}{n})\;.
$$
\end{proposition}

\begin{proof}

We will use the Efron-Stein inequality to show that $\var{\norm{\betaHat}^2}$ goes to 0 as $n\tendsto \infty$. In what follows, we assume that $\psi(\eps_i)$ have enough moments for all the expectations of the type $\Exp{\norm{\betaHat}^{2k}}$ to be bounded like $1/\tau^{2k}$. Note that this the content of our Lemma \ref{lemma:betaHatIsBounded}.

Recall that the Efron-Stein inequality \citet{EfronStein81} gives, if $Y$ is a function of $n$ independent random variables, and $Y_{(i)}$ is any function of all those random variables except the $i$-th, 
$$
\var{Y}\leq \sum_{i=1}^n \var{Y-Y_{(i)}}\leq \sum_{i=1}^n \Exp{(Y-Y_{(i)})^2}\;.
$$

We first observe that
$$
\Exp{|\norm{\betaHat}^2-\norm{\betaHat_{(i)}}^2|^2}\leq 2\left[\Exp{|\norm{\betaHat}^2-\norm{\betaTilde_{i}}^2|^2}+\Exp{|\norm{\betaTilde_i}^2-\norm{\betaHat_{(i)}}^2|^2}\right]\;.
$$
Of course, using the fact that $\betaHat=\betaHat-\betaTilde_i+\betaTilde_i$ and $|\norm{\betaHat}^2-\norm{\betaTilde_{i}}^2|^2=[(\betaHat-\betaTilde_{i})\trsp(\betaHat+\betaTilde_i)]^2$, $|\norm{\betaHat}^2-\norm{\betaTilde_{i}}^2|^2=\gO_{L_1}(\norm{\betaHat-\betaTilde_{i}}^4)+\sqrt{\gO_{L_1}(\norm{\betaHat-\betaTilde_{i}}^4)}$, by the Cauchy-Schwarz inequality, since  $\Exp{\norm{\betaHat}^2}$ exists and is bounded by $K/\tau^2$. 

Using the results of Corollary \ref{coro:AggregResApproxBetaHatByBetaTildeIncl}, we see that
$$
\Exp{|\norm{\betaHat}^2-\norm{\betaTilde_{i}}^2|^2}=\gO(\frac{\polyLog(n)}{n^2})=\lo(n^{-1})\;.
$$

On the other hand, given the definition in Equation \eqref{eq:ApproxBetaHatStandardLeaveOneOut},
$$
\norm{\betaTilde_i}^2-\norm{\betaHat_{(i)}}^2=2 \frac{1}{n}\betaHat_{(i)}\trsp (S_i+\tau \id)^{-1}X_i \psi(\prox_{c_i}(\rTilde_{i,(i)}))
+\frac{1}{n^2}X_i\trsp (S_i+\tau \id)^{-2}X_i \psi^2(\prox_{c_i}(\rTilde_{i,(i)}))\;.
$$
Since $S_i$ is independent of $X_i$, and $\norm{(S_i+\tau\id)^{-1}}\leq 1/\tau$, $\betaHat_{(i)}\trsp (S_i+\tau \id)^{-1}X_i=\gO_{L_4}(\norm{\betaHat_{(i)}}/c_n^{1/2})$, using our concentration assumptions applied to linear forms. Therefore, we see that both terms are $\gO_{L_2}(1/nc_n^{1/2})$ provided $\psi(\prox_{c_i}(\rTilde_{i,(i)}))$ has $4+\eps$ absolute moments - uniformly bounded in $n$ -  by using H\"older's inequality. Under our assumptions, given our work on $\rTilde_{i,(i)}$, the fact that the prox is a contractive mapping (\cite{MoreauProxPaper65}) and that we assume that $\sgn(\psi(x))=\sgn(x)$, it is clear that this is the case. We conclude that then
$$
\Exp{\left|\norm{\betaTilde_i}^2-\norm{\betaHat_{(i)}}^2\right|^2}=\gO(\frac{1}{n^2c_n})=\gO(\frac{\polyLog(n)}{n^2})\;.
$$

Taking $Y=\norm{\betaHat}^2$ and $Y_{(i)}=\norm{\betaHat_{(i)}}^2$ in the Efron-Stein inequality, we clearly see that
$$
\var{\norm{\betaHat}^2}=\gO(\frac{\polyLog(n)}{n})=\lo(1)\;.
$$
This shows that  $\norm{\betaHat}$ has a deterministic equivalent in probability and in $L_2$.
\end{proof}
\section{Leaving out a predictor}
In \cite{NEKRobustPaperPNAS2013Published}, we showed through probabilistic heuristics that the probabilistic properties of the entries of $\betaHat$ could be understood by leaving out predictors. We now show that all the formal manipulations we did in that paper are valid under our assumptions. In that step, we do need at various points that the entries of the data vector $X_i$ be independent, whereas as we showed before, it is not important when studying what happens when we leave out an observation. 

We call $V$ the $n\times (p-1)$ matrix corresponding to the first $(p-1)$ columns of the design matrix $X$. We call $V_i$ in $\mathbb{R}^{p-1}$ the vector corresponding to the first $p-1$ entries of $X_i$, i.e $V_i\trsp =(X_i(1),\ldots,X_i(p-1))$. 

Let us call $\gammaHat$ the solution of our optimization problem when $X_i(p)=0$ for all $i$, i.e the solution we get when we solve our original problem with the design matrix $V$ instead of $X$. 
\newcommand{\ansatzBetap}{\mathfrak{b}_p}
\newcommand{\bTilde}{\widetilde{b}}

The corresponding residuals are $\{r_{i,[p]}\}_{i=1}^n$. Hence, $r_{i,[p]}=\eps_i-V_i\trsp \gammaHat$.
We call
\begin{align*}
u_p&=\frac{1}{n}\sum_{i=1}^n \psi'(r_{i,[p]}) V_i X_i(p)\;,\\
\mathfrak{S}_p&=\frac{1}{n}\sum_{i=1}^n \psi'(r_{i,[p]}) V_i V_i\trsp\;.
\end{align*}

Note that $\mathfrak{S}_p$ is $(p-1)\times (p-1)$. We call 
\begin{equation}\label{eq:definitionXin}
\xi_n\triangleq \frac{1}{n}\sum_{i=1}^n X_i^2(p)\psi'(r_{i,[p]})-u_p\trsp (\mathfrak{S}_p+\tau \id)^{-1}u_p\;,
\end{equation}
and 
\begin{equation}\label{eq:definitionNp}
N_p\triangleq \frac{1}{\sqrt{n}}\sum_{i=1}^n X_i(p)\psi(r_{i,[p]})\;.
\end{equation}

We consider
\begin{equation}\label{eq:propositionAnsatzBetaHatp}
\ansatzBetap\triangleq 	\frac{1}{\sqrt{n}} \frac{N_p}{\tau+\xi_n}\;.
\end{equation}

We will show later, in Subsubsection \ref{subsubsec:onXin} that $\xi_n\geq 0$. Note that when $\xi_n>0$, we have 
$$
\ansatzBetap=\frac{\frac{1}{n}\sum_{i=1}^n X_i(p)\psi(r_{i,[p]})-\tau \ansatzBetap}{\frac{1}{n}\sum_{i=1}^n X_i^2(p)\psi'(r_{i,[p]})-u_p\trsp (\mathfrak{S}_p+\tau \id)^{-1}u_p}= \frac{n^{-1/2}N_p-\tau \ansatzBetap}{\xi_n}\;.
$$

We call
\begin{equation}\label{eq:AnsatzBetaHatFromGammaHat}
\bTilde=\begin{bmatrix}
\gammaHat\\ 0
\end{bmatrix}+\ansatzBetap \begin{bmatrix}
-(\mathfrak{S}_p+\tau \id)^{-1} u_p\\
1
\end{bmatrix}\;.
\end{equation}
The aim of our work is to establish Corollary \ref{coro:ApproxBetaHatThroughLeaveOnePredictorOut}, which shows that $\bTilde$ is a $\sqrt{n}$-consistent approximation of $\betaHat$ - in the sense of Euclidian norm. Because the last coordinate of $\bTilde$ has a reasonably simple probabilistic structure and our approximations are sufficiently good, we will be able to transfer our insights about this coordinate to $\betaHat_p$. 

Once again, the approximating quantities we consider are ``very natural'' in light of our work in \cite{NEKRobustPaperPNAS2013Published}.

\subsection{Deterministic aspects}
\begin{proposition}
We have 
\begin{equation}\label{eq:ControlBTilde}
\norm{\betaHat-\bTilde}\leq \frac{1}{\tau} |\ansatzBetap|\sup_{1\leq i\leq n}|\mathsf{d}_{i,p}| \, \opnorm{\SigmaHat} \sqrt{\norm{(\mathfrak{S}_p+\tau\id)^{-1}u_p}^2+1}\;.
\end{equation}
where $\mathsf{d}_{i,p}=[\psi'(\gamma^*_{i,p})-\psi'(r_{i,[p]})]$ and $\gamma^*_{i,p}$ is in the interval $(\eps_i-V_i\trsp \gammaHat,\eps_i-X_i\trsp \bTilde)$.

Furthermore, $\norm{(\mathfrak{S}_p+\tau\id)^{-1}u_p}^2\leq\frac{1}{n}\sum_{i=1}^n X_i^2(p)\psi'(r_{i,[p]}).$
\end{proposition}

As we saw in Equation \eqref{eq:exactRelationDeltafDeltaBeta}, we have
$$
\norm{\betaHat-\bTilde}\leq \frac{1}{\tau}\norm{f(\bTilde)}\;,
$$
where
$$
f(\bTilde)=-\frac{1}{n}\sum_{i=1}^n X_i \psi(\eps_i-X_i\trsp \bTilde)+\tau \bTilde\;.
$$
We note furthermore that
$$
g(\gammaHat)\triangleq -\frac{1}{n}\sum_{i=1}^n V_i \psi(\eps_i-V_i\trsp \gammaHat)+\tau \gammaHat=0_{p-1}\;.
$$

The strategy of the proof is to control $f(\bTilde)$ by approximating it by $g(\gammaHat)$.

\begin{proof}

\textbf{{\small a) Work on the first $(p-1)$ coordinates of $f(\bTilde)$}}\\
We call $\mathsf{f}_{p-1}(\beta)$ the first $p-1$ coordinates of $f(\beta)$. We call $\gammaHat_{ext}$ the $p$-dimensional vector whose first $p-1$ coordinates are $\gammaHat$ and last coordinate is 0, i.e 
$$
\gammaHat_{ext}=\begin{bmatrix}
\gammaHat\\ 0
\end{bmatrix}\;.
$$

For a vector $v$, we use the notation $v_{comp,k}$ to denote the $p-1$ dimensional vector consisting of all the coordinates of $v$ except the $k$-th.

Clearly,
$$
\mathsf{f}_{p-1}(\bTilde)=\mathsf{f}_{p-1}(\beta)-g(\gammaHat)=-\frac{1}{n}\sum_{i=1}^n V_i \left[\psi(\eps_i-X_i\trsp \bTilde)-\psi(\eps_i-V_i\trsp \gammaHat)\right]+\tau(\bTilde_{comp,p}-\gammaHat)\;.
$$
We can write by using the mean value theorem
$$
\psi(\eps_i-X_i\trsp \bTilde)-\psi(\eps_i-V_i\trsp \gammaHat)=\psi'(r_{i,[p]})X_i\trsp (\gammaHat_{ext}-\bTilde)+[\psi'(\gamma^*_{i,p})-\psi'(r_{i,[p]})]X_i\trsp (\gammaHat_{ext}-\bTilde)
$$
Let us call
\begin{align*}
\mathsf{d}_{i,p}&=[\psi'(\gamma^*_{i,p})-\psi'(r_{i,[p]})]\;,\\
\delta_{i,p}&=[\psi'(\gamma^*_{i,p})-\psi'(r_{i,[p]})]X_i\trsp (\gammaHat_{ext}-\bTilde)\;,\\
\mathsf{R}_p&=-\frac{1}{n}\sum_{i=1}^n \mathsf{d}_{i,p} V_i X_i\trsp (\gammaHat_{ext}-\bTilde)\;.
\end{align*}

We have with this notation
$$
\mathsf{f}_{p-1}(\bTilde)=-\frac{1}{n}\sum_{i=1}^n \psi'(r_{i,[p]}) V_i X_i\trsp (\gammaHat_{ext}-\bTilde)+\tau(\bTilde_{comp,p}-\gammaHat)+\mathsf{R}_p\triangleq \mathsf{A}_p+\mathsf{R}_p\;.
$$

We note that by definition,
\begin{align*}
\gammaHat_{ext}-\bTilde &=\ansatzBetap \begin{bmatrix} (\mathfrak{S}_p+\tau \id)^{-1} u_p \\-1\end{bmatrix}\;, \\
\bTilde_{comp,p}-\gammaHat &=-\ansatzBetap (\mathfrak{S}_p+\tau \id)^{-1} u_p\;.
\end{align*}

Therefore,
$$
\mathsf{A}_p=-\ansatzBetap\left(\frac{1}{n}\sum_{i=1}^n \psi'(r_{i,[p]})V_i\left[V_i\trsp (\mathfrak{S}_p+\tau \id)^{-1} u_p -X_i(p)\right]\right)+
\tau (-\ansatzBetap (\mathfrak{S}_p+\tau \id)^{-1} u_p)\;.
$$
Recalling the definition of $\mathfrak{S}_p$ and $u_p$, we see that
$$
\mathsf{A}_p=-\ansatzBetap\left(\mathfrak{S}_p(\mathfrak{S}_p+\tau \id)^{-1} u_p -u_p+\tau (\mathfrak{S}_p+\tau \id)^{-1} u_p\right)=0_{p-1}\;,
$$
since $\mathfrak{S}_p(\mathfrak{S}_p+\tau \id)^{-1}+\tau (\mathfrak{S}_p+\tau \id)^{-1}=\id$.

We conclude that
$$
\boxed{
\mathsf{f}_{p-1}(\bTilde)=\mathsf{R}_p\;.
}
$$

\textbf{{\small b) Work on the last coordinate of $f(\bTilde)$}}\\
We call $[f(\bTilde)]_p$ the last coordinate of $f(\bTilde)$. We recall the representation
$$
\psi(\eps_i-X_i\trsp \bTilde)-\psi(\eps_i-V_i\trsp \gammaHat)=\psi'(r_{i,[p]})X_i\trsp (\gammaHat_{ext}-\bTilde)+[\psi'(\gamma^*_{i,p})-\psi'(r_{i,[p]})]X_i\trsp (\gammaHat_{ext}-\bTilde)
$$
and call
$$
\delta_{i,p}=[\psi'(\gamma^*_{i,p})-\psi'(r_{i,[p]})]X_i\trsp (\gammaHat_{ext}-\bTilde)\;.
$$
Clearly,
\begin{align*}
\psi(\eps_i-X_i\trsp \bTilde)&=\psi(r_{i,[p]})+\psi'(r_{i,[p]})X_i\trsp (\gammaHat_{ext}-\bTilde)+\delta_{i,p}\;,\\
&=\psi(r_{i,[p]})+\psi'(r_{i,[p]})\ansatzBetap \left[V_i\trsp(\mathfrak{S}_p+\tau \id)^{-1} u_p-X_i(p)\right]+\delta_{i,p}\;.
\end{align*}
We therefore see that
\begin{align*}
[f(\bTilde)]_p+\frac{1}{n}\sum_{i=1}^n X_i(p)\delta_{i,p}&=-\frac{1}{n} \sum_{i=1}^n X_i(p)\left(\psi(r_{i,[p]})+\psi'(r_{i,[p]})\ansatzBetap \left[V_i\trsp(\mathfrak{S}_p+\tau \id)^{-1} u_p-X_i(p)\right]\right) +\tau \bTilde_p\;,\\
&=-\frac{1}{n}\sum_{i=1}^n X_i(p)\psi(r_{i,[p]})-\ansatzBetap u_p\trsp (\mathfrak{S}_p+\tau \id)^{-1} u_p +\ansatzBetap \frac{1}{n}\sum_{i=1}^n \psi'(r_{i,[p]}) X_i^2(p) +\tau \ansatzBetap\;,\\
&=-\left[\frac{1}{n}\sum_{i=1}^n X_i(p)\psi(r_{i,[p]})-\tau \ansatzBetap\right]+\ansatzBetap \left(\frac{1}{n}\sum_{i=1}^n \psi'(r_{i,[p]}) X_i^2(p)-u_p\trsp (\mathfrak{S}_p+\tau \id)^{-1} u_p\right)\;,\\
&=0\;.
\end{align*}

We conclude that
$$
[f(\bTilde)]_p=-\frac{1}{n}\sum_{i=1}^n X_i(p)\delta_{i,p}=-\frac{1}{n}\sum_{i=1}^n \mathsf{d}_{i,p} X_i(p) X_i\trsp (\gammaHat_{ext}-\bTilde)\;.
$$

\subsubsection*{Representation of $f(\bTilde)$}
Aggregating all the results we have obtained so far, we see that
\begin{align*}
f(\bTilde)&=\left(-\frac{1}{n}\sum_{i=1}^n \mathsf{d}_{i,p} X_i X_i\trsp\right)(\gammaHat_{ext}-\bTilde)\;,\\
&=\ansatzBetap\left(\frac{1}{n}\sum_{i=1}^n \mathsf{d}_{i,p} X_i X_i\trsp\right)
\begin{bmatrix} (\mathfrak{S}+\tau\id)^{-1}u_p\\ -1
\end{bmatrix}
\;.
\end{align*}

We conclude immediately that
\begin{equation}\label{eq:boundfBtilde}
\norm{f(\bTilde)}\leq |\ansatzBetap|\sup_{1\leq i\leq n}|\mathsf{d}_{i,p}| \, \opnorm{\SigmaHat} \sqrt{\norm{(\mathfrak{S}+\tau\id)^{-1}u_p}^2+1}\;.
\end{equation}
\newcommand{\DPsiPrimeRDotp}{D_{\psi'(r_{\cdot,[p]})}}
Calling $\DPsiPrimeRDotp$ the diagonal matrix with $(i,i)$ entry $\psi'(r_{i,[p]})$, we see that
$$
u_p=\frac{1}{n}V\trsp \DPsiPrimeRDotp X(p)\;.
$$
Therefore,
$$
\norm{({\mathfrak S}+\tau\id)^{-1}u_p}^2=\frac{X(p)}{\sqrt{n}}\DPsiPrimeRDotp^{1/2}  \frac{\DPsiPrimeRDotp^{1/2}V}{\sqrt{n}}\left(\frac{V\trsp \DPsiPrimeRDotp V}{n}+\tau \id\right)^{-1}\frac{V\trsp \DPsiPrimeRDotp^{1/2}}{\sqrt{n}}\DPsiPrimeRDotp^{1/2} \frac{X(p)}{\sqrt{n}}\;.
$$

Clearly, 
$$
\frac{\DPsiPrimeRDotp^{1/2}V}{\sqrt{n}}\left(\frac{V\trsp\DPsiPrimeRDotp V}{n}+\tau \id\right)^{-1}\frac{V\trsp \DPsiPrimeRDotp^{1/2}}{\sqrt{n}}
\preceq \id\;.
$$
So we have 
\begin{equation}\label{eq:boundgothicSPlusTauIdInverse}
\norm{({\mathfrak S}+\tau\id)^{-1}u_p}^2\leq \frac{1}{n}X(p)\trsp \DPsiPrimeRDotp X(p)=\frac{1}{n}\sum_{i=1}^n X_i^2(p)\psi'(r_{i,[p]})\;.
\end{equation}
\end{proof}

\subsection{Probabilistic aspects}

From now on, we assume that $X(p)$, the $p$-th column of the design matrix, is independent of $\{V_i,\eps_i\}_{i=1}^n$.

Because $r_{i,[p]}$ are the residuals from a ``full model" with $p-1$ predictors, the analysis done above concerning the $R_i$ - see Lemma \ref{lemma:BoundRiFromrTildeii} - applies and will allow us to control $\max_{1\leq i\leq n}|\psi'(r_{i,[p]})|^2$. (Note that the distribution of the errors is the same whether we use $p$ or $p-1$ predictors because we assume in the regression model that $\beta_0=0$ - the study of ridge-regularized robust regression would require an adjustment in the non-null case where $\beta_0\neq 0$, but since we limit ourselves to the null case, no such adjustment is needed.)

In light of Lemma \ref{lemma:BoundRiFromrTildeii} and using independence of $X_i(p)$'s and $r_{i,[p]}$, it is clear that the upper bound in Equation \eqref{eq:boundgothicSPlusTauIdInverse} is $O_{L_k}(\polyLog(n))$. 

Hence,
$$
\norm{({\mathfrak S}_p+\tau\id)^{-1}u_p}^2=\gO_{L_k}(\polyLog(n))
$$

This guarantees that
$$
\begin{Vmatrix} (\mathfrak{S}_p+\tau\id)^{-1}u_p\\ -1
\end{Vmatrix}^2\leq (1+\norm{(\mathfrak{S}_p+\tau\id)^{-1}u_p}^2)=\gO_{L_k}(\polyLog(n))\;.
$$

We conclude, using Equation \eqref{eq:boundfBtilde}, that
$$
f(\bTilde)\leq K \polyLog(n) |\ansatzBetap| \sup_{1\leq i\leq n}|\mathsf{d}_{i,p}| \, \opnorm{\SigmaHat} \text{ in } L_k.
$$

At a high level, we expect $\sup_{1\leq i\leq n}|\mathsf{d}_{i,p}|$ to be small, even compared to $\max_{1\leq i\leq n}|\psi'(r_{i,[p]})|$ which should give us that
$$
f(\bTilde)=\lo_{L_k}(\polyLog(n)|\ansatzBetap|)\;.
$$
We now show that this latter quantity is small.
\subsubsection{On $\ansatzBetap$}

We recall the notations
\begin{align*}
N_p&=\frac{1}{\sqrt{n}} \sum_{i=1}^n \psi(r_{i,[p]}) X_i(p)\;,\\
\xi_n&=\frac{1}{n}\sum_{i=1}^n \psi'(r_{i,[p]}) X_i^2(p) -u_p\trsp (\mathfrak{S}_p+\tau \id)^{-1}u_p\;.
\end{align*}

Under our assumptions, we have $\Exp{X_i}=0$ and $\scov{X_i}=\id_p$ and hence $\Exp{X_i^2(p)}=1$. Recall that since we assume that $X(p)$ is independent of $\{V_i,\eps_i\}_{i=1}^n$, $X(p)$ is independent of $\{r_{i,[p]}\}_{i=1}^n$.

\begin{proposition}
We have 
$$
|\ansatzBetap|\leq \frac{1}{\sqrt{n}\tau}|N_p|\;.
$$
Furthermore, under our assumptions, $N_p=\gO_{L_k}(\polyLog(n))$ and therefore 
$$
\ansatzBetap=\gO_{L_k}(\polyLog(n) n^{-1/2})\;.
$$
\end{proposition}
\begin{proof}
From the definition of $\ansatzBetap$, we see that, when $\xi_n\neq 0$
$$
\ansatzBetap=\frac{1}{\sqrt{n}}\frac{N_p}{\tau+\xi_n}\;.
$$
We will see later, in Subsubsection \ref{subsubsec:onXin}, that $\xi_n\geq 0$. It immediately then follows that
$$
\left|\ansatzBetap\right|\leq \frac{1}{\sqrt{n}\tau}|N_p|\;.
$$
Using independence of $X(p)$ and $\{V_i,\eps_i\}_{i=1}^n$, we have

$$
\Exp{N_p^2}=\frac{1}{n}\sum_{i=1}^n\Exp{X_i^2(p)}\Exp{\psi^2(r_{i,[p]})}\;,
$$
whether the right-hand side is finite or not. 

Since $r_{i,[p]}$ are the residuals of the full model with $p-1$ predictors, our previous analyses show that $N_p$ has as many moments as we need and 
$N_p=\gO_{L_k}(\polyLog(n))$. (Indeed, it suffices to apply reasoning similar to the arguments given in Lemma \ref{lemma:betaHatIsBounded} for the control of the moments and our bounds on $r_{i,[p]}$ and therefore on $\psi(r_{i,[p]})$)

We therefore have
$$
|\ansatzBetap|\leq \frac{1}{\sqrt{n}\tau}\gO_{L_k}(\polyLog(n))\;.
$$
\end{proof}
\subsubsection{On $\xi_n$}\label{subsubsec:onXin}
Let us write $\xi_n$ in matrix form: denoting by $X(p)$ the last column of the design matrix $X$, we have 
\begin{equation}\label{eq:defXinInMatrixForm}
\xi_n=\frac{1}{n}X(p)\trsp \DPsiPrimeOneLessPred^{1/2} M \DPsiPrimeOneLessPred^{1/2} X(p)\;,
\end{equation}
where
\begin{equation}\label{eq:definitionMatrixM}
M=\id_n-\frac{\DPsiPrimeOneLessPred^{1/2} V}{\sqrt{n}}\left(\frac{1}{n}V\trsp \DPsiPrimeOneLessPred V+\tau \id\right)^{-1}\frac{V\trsp \DPsiPrimeOneLessPred^{1/2}}{\sqrt{n}}\;.
\end{equation}

\begin{lemma}
We have 
$$
\xi_n\geq 0\;.
$$

Furthermore,  
\begin{equation}\label{eq:BoundDeviationXin}
|\xi_n-\frac{1}{n}\trace{\DPsiPrimeOneLessPred^{1/2} M\DPsiPrimeOneLessPred^{1/2}}|=\gO_{L_k}(\sup_{1\leq i \leq n}\psi'(r_{i,[p]})/(\sqrt{nc_n}))\;.
\end{equation}
\end{lemma}
\begin{proof}
Let us first focus on 
$$
M=\id_n-\frac{1}{n}\DPsiPrimeOneLessPred^{1/2}V(\frac{V\trsp \DPsiPrimeOneLessPred V}{n}+\tau \id)^{-1}V\trsp \DPsiPrimeOneLessPred^{1/2}\;.
$$
When $\tau>0$, it is clear that all the eigenvalues of $M$ are strictly positive, i.e $M$ is positive definite. Indeed, if the singular values of $n^{-1/2}\DPsiPrimeOneLessPred^{1/2}V$ are denoted by $\sigma_i$, the eigenvalues of $M$ are $\tau/(\sigma_i^2+\tau)$. 

Therefore, since $\xi_n=\frac{1}{n}v\trsp M v$ with $v=\DPsiPrimeOneLessPred^{1/2}X(p)$, $\xi_n\geq 0$. 

As we have seen above, $M$ has eigenvalues between $0$ and 1. Therefore, 
$$
0\preceq \DPsiPrimeOneLessPred^{1/2} M \DPsiPrimeOneLessPred^{1/2}\preceq \DPsiPrimeOneLessPred\;.
$$

The matrix $M$ is independent of $X(p)$. $\DPsiPrimeOneLessPred$ is also independent of $X(p)$. 

We can now appeal to Lemma \ref{lemma:RandomQuadFormsBounds} to obtain 
$$
\left|\frac{1}{n}X(p)\trsp \DPsiPrimeOneLessPred^{1/2} M \DPsiPrimeOneLessPred^{1/2} X(p)-\frac{1}{n}\trace{\DPsiPrimeOneLessPred^{1/2} M \DPsiPrimeOneLessPred^{1/2}}\right|=\gO_{L_k}(\frac{1}{\sqrt{nc_n}}\sup_{i}\psi'(r_{i,[p]}))\;.
$$
\end{proof}
\subsubsection*{About $\frac{1}{n}\trace{D^{1/2}_{\psi'(r_{\cdot,[p]})} M D^{1/2}_{\psi'(r_{\cdot,[p]})}}$}
\begin{lemma}\label{lemma:approxXinByTrace}
Let us call $\mathfrak{S}_{p}=\frac{1}{n}\sum_{i=1}^n \psi'(r_{i,[p]})V_iV_i\trsp$ and $\mathfrak{S}_p(i)=\mathfrak{S}_p-\frac{1}{n}\psi'(r_{i,[p]})V_iV_i\trsp$.
Let us also call
\begin{align*}
\mathsf{c}_{\tau,p}&=\frac{1}{n}\trace{(\mathfrak{S}_{p}+\tau \id)^{-1}}\;,\\
\eta_i&=\frac{1}{n}V_i\trsp (\mathfrak{S}_{p}(i)+\tau \id)^{-1}V_i-\mathsf{c}_{\tau,p}\;.
\end{align*}
Then we have
\begin{equation}\label{eq:ApproxTraceMD}
\left|\frac{1}{n}\trace{\id_n-M}-\left(\frac{1}{n}\trace{D^{1/2}_{\psi'(r_{\cdot,[p]})} M D^{1/2}_{\psi'(r_{\cdot,[p]})}}\right)\mathsf{c}_{\tau,p}\right|\leq \left[\sup_{i} |\eta_i|\right] \frac{1}{n}\sum_{i}\psi'(r_{i,[p]})\;.
\end{equation}
We also have 
$$
\frac{1}{n}\trace{\id_n-M}=\frac{p}{n}-\tau\mathsf{c}_{\tau,p}\;.
$$	
\end{lemma}
\begin{proof}
We call $d_{i,i}=\psi'(r_{i,[p]})/n$. Of course, by using the Sherman-Morrison-Woodbury formula (see e.g \cite{hj}, p.19), 
\begin{align*}
M_{i,i}&=1-d_{i,i}V_i\trsp(V\trsp \DPsiPrimeOneLessPred V/n+\tau \id)^{-1} V_i\;,\\
&=1-d_{i,i} \frac{V_i\trsp (\mathfrak{S}_{p}(i)+\tau \id)^{-1}V_i}{1+d_{i,i} V_i\trsp (\mathfrak{S}_{p}(i)+\tau \id)^{-1}V_i}\;,\\
&=\frac{1}{1+d_{i,i} V_i\trsp (\mathfrak{S}_{p}(i)+\tau \id)^{-1}V_i}\;.
\end{align*}
Recall that we are interested in $\frac{1}{n}\sum_i \psi'(r_{i,[p]})M_{i,i}=\frac{1}{n}\trace{D^{1/2}_{\psi'(r_{\cdot,[p]})} M D^{1/2}_{\psi'(r_{\cdot,[p]})}}$. Note that
$$
\trace{\id_n-M}=\trace{\mathfrak{S}_{p}(\mathfrak{S}_{p}+\tau \id)^{-1}}=p-\tau\trace{(\mathfrak{S}_{p}+\tau \id)^{-1}}=p-n\tau \mathsf{c}_{\tau,p}\;.
$$
On the other hand,
\begin{equation}\label{eq:keyEqnInImplicitDefOfCtaup}
\trace{\id_n-M}=\sum_{i}(1-M_{i,i})=\sum_{i}d_{i,i} \frac{V_i\trsp (\mathfrak{S}_{p}(i)+\tau \id)^{-1}V_i}{1+d_{i,i} V_i\trsp (\mathfrak{S}_{p}(i)+\tau \id)^{-1}V_i}\;.
\end{equation}

With our definitions, we have 
$$
\frac{1}{n}\trace{\id_n-M}=\left(\frac{1}{n}\sum_i \psi'(r_{i,[p]})M_{i,i}\right)\mathsf{c}_{\tau,p}+\frac{1}{n}\sum_{i}\psi'(r_{i,[p]}) \frac{\eta_i}{1+d_{i,i} V_i\trsp (\mathfrak{S}_{p}(i)+\tau \id)^{-1}V_i}\;.
$$
It immediately follows that 
$$
\left|\frac{1}{n}\trace{\id_n-M}-\left(\frac{1}{n}\sum_i \psi'(r_{i,[p]})M_{i,i}\right)\mathsf{c}_{\tau,p}\right|\leq \left[\sup_{i} |\eta_i|\right] \frac{1}{n}\sum_{i}\psi'(r_{i,[p]})\;,
$$
as announced.
\end{proof}
\subsubsection*{Controlling $\eta_i$}
\begin{lemma}\label{lemma:ControlEtais}
Suppose we can find $\{\mathsf{r}^{(i)}_{j,[p]}\}_{j\neq i}$ independent of $V_i$ such that $\sup_{j\neq i}|\mathsf{r}^{(i)}_{j,[p]}-r_{j,[p]}|\leq \delta_n(i)$. Leaving out $V_i$ from a regression comes of course to mind and the work of the first section will apply.

Suppose further that we can find $K_n$ such that
$$
\sup_i\sup_{j\neq i}|\psi'(\mathsf{r}^{(i)}_{j,[p]})-\psi'(r_{j,[p]})|\leq K_n
$$
Then 
\begin{equation}\label{eq:ControlSupEtais}
\sup_i |\eta_i|=\gO_{L_k}\left(
\frac{1}{\tau^2} K_n \opnorm{\SigmaHat}+\frac{\polyLog(n)}{\sqrt{nc_n}}+\frac{1}{n\tau}\right)\;,
\end{equation}
provided $K_n$ has $3k$ uniformly bounded moments. 
\end{lemma}
\begin{proof}
We call
$$
AM_{i,p}=\frac{1}{n}\sum_{j\neq i}\psi'(\mathsf{r}^{(i)}_{j,[p]}) V_jV_j\trsp\;.
$$
	
Then, using for instance the first resolvent identity, i.e $A^{-1}-B^{-1}=A^{-1}(B-A) B^{-1}$, we see that
$$
\opnorm{(\mathfrak{S}_{p}(i)+\tau \id)^{-1}-(AM_{i,p}+\tau \id)^{-1}}\leq \frac{1}{\tau^2} K_n \opnorm{\SigmaHat}\;.
$$
In particular,
$$
\left|\frac{1}{n}V_i\trsp (\mathfrak{S}_{p}(i)+\tau \id)^{-1}V_i-\frac{1}{n}V_i\trsp (AM_{i,p}+\tau\id)^{-1}V_i\right|\leq \frac{\norm{V_i}^2}{n}\frac{1}{\tau^2} K_n \opnorm{\SigmaHat}\;.
$$
However, since $AM_{i,p}$ is independent of $V_i$, we can use Lemma \ref{lemma:RandomQuadFormsBounds} and see that
$$
\sup_{1\leq i \leq n}\left|\frac{1}{n}V_i\trsp (AM_{i,p}+\tau\id)^{-1}V_i-\frac{1}{n}\trace{(AM_{i,p}+\tau\id)^{-1}}\right|=\gO_{L_k}(\frac{\polyLog(n)}{\sqrt{nc_n}})\;,
$$
by using the fact that $\lambda_{\max}((AM_{i,p}+\tau\id)^{-1})\leq\frac{1}{\tau}$.

However, by the argument we gave above, 
$$
\left|\frac{1}{n}\trace{(AM_{i,p}+\tau\id)^{-1}}-\frac{1}{n}\trace{(\mathfrak{S}_{p}(i)+\tau \id)^{-1}}\right|\leq \frac{1}{\tau^2} K_n \opnorm{\SigmaHat} \frac{p}{n}\;.
$$

We conclude that
$$
\sup_{1\leq i \leq n} \left|\frac{1}{n}V_i\trsp (\mathfrak{S}_{p}(i)+\tau \id)^{-1}V_i-\frac{1}{n}\trace{(\mathfrak{S}_{p}(i)+\tau \id)^{-1}}\right|\leq
\frac{1}{\tau^2} K_n \opnorm{\SigmaHat} \sup_{1\leq i \leq n} \left[\frac{p}{n}+\frac{\norm{V_i}^2}{n}\right]+\frac{\polyLog(n)}{\sqrt{nc_n}}\;,
$$
in $L_k$.

Now, it is clear that $\sup_{1\leq i \leq n}\norm{V_i}^2/n=\gO_{L_k}(1)$ and finally
$$
\sup_{1\leq i \leq n} \left|\frac{1}{n}V_i\trsp (\mathfrak{S}_{p}(i)+\tau \id)^{-1}V_i-\frac{1}{n}\trace{(\mathfrak{S}_{p}(i)+\tau \id)^{-1}}\right|
=\gO_{L_k}(\frac{1}{\tau^2} K_n \opnorm{\SigmaHat}+\frac{\polyLog(n)}{\sqrt{nc_n}})\;.
$$

{\small \textbf{Control of $\frac{1}{n}\trace{(\mathfrak{S}_{p}(i)+\tau \id)^{-1}}-\frac{1}{n}\trace{(\mathfrak{S}_{p}+\tau \id)^{-1}}$}}\\
Using the Sherman-Woodbury-Morrison formula, we have
$$
(\mathfrak{S}_{p}(i)+\tau \id)^{-1}-(\mathfrak{S}_{p}+\tau \id)^{-1}=\frac{\psi'(r_{i,[p]})}{n}\frac{(\mathfrak{S}_{p}(i)+\tau \id)^{-1}V_i V_i\trsp (\mathfrak{S}_{p}(i)+\tau \id)^{-1}}{1+\frac{\psi'(r_{i,[p]})}{n}V_i\trsp (\mathfrak{S}_{p}(i)+\tau \id)^{-1}V_i}\;.
$$
After taking traces, we see that
$$
0\leq \trace{(\mathfrak{S}_{p}(i)+\tau \id)^{-1}}-\trace{(\mathfrak{S}_{p}+\tau \id)^{-1}}\leq \frac{1}{\tau}\;,
$$
since $V_i\trsp (\mathfrak{S}_{p}(i)+\tau \id)^{-2}V_i\leq \frac{1}{\tau}V_i\trsp (\mathfrak{S}_{p}(i)+\tau \id)^{-1}V_i$.

Therefore,
$$
0\leq \frac{1}{n}\trace{(\mathfrak{S}_{p}(i)+\tau \id)^{-1}}-\frac{1}{n}\trace{(\mathfrak{S}_{p}+\tau \id)^{-1}}\leq \frac{1}{n\tau}\;.
$$

We conclude that
$$
\sup_{1\leq i \leq n}\left|\eta_i\right|=\gO_{L_k}\left(
\frac{1}{\tau^2} K_n \opnorm{\SigmaHat}+\frac{\polyLog(n)}{\sqrt{nc_n}}+\frac{1}{n\tau}\right)\;,
$$
provided we can use Holder's inequality. In effect, this requires $K_n$ to have $3k$ uniformly bounded moments. 
\end{proof}
\subsubsection{Control of $K_n$}
A natural choice for $\mathsf{r}^{(i)}_{j,[p]}$ defined in Lemma \ref{lemma:ControlEtais} is to use a leave one out estimator of $\gammaHat$. Hence, all the work done in 
Corollary \ref{coro:AggregResApproxBetaHatByBetaTildeIncl} becomes immediately relevant. 
\begin{lemma}\label{lemma:ControlKn}
With the notations of Lemma \ref{lemma:ControlEtais}, we have 
$$
\sup_{i} (\delta_n(i))=\gO_{L_k}\left(\frac{\polyLog(n)}{n^{1/2}}\right)\;.
$$
Therefore, 
$$
K_n=\gO_{L_k}\left(n^{-1/2}\polyLog(n)\right) 
$$
\end{lemma}
\begin{proof}
The first statement of the Lemma is an application of Corollary \ref{coro:AggregResApproxBetaHatByBetaTildeIncl} with $R_j=r_{j,[p]}$ and $\rTilde_{j,(i)}=\mathsf{r}^{(i)}_{j,[p]}$.

The control of $K_n$ follows immediately by using our assumptions on $\psi'$ and on the growth of ${\mathcal B}_n(i)$ and $L({\mathcal B}_n(i))$ we had before, now applied to the situation with $p-1$ predictors. 
\end{proof}

\textbf{Important remark:} the previous remark has important consequences for $c_i$ defined in Equation \eqref{eq:defciandetaifirstpart}: we just showed that $\sup_{i}|\frac{1}{n}V_i\trsp (\mathfrak{S}_p(i)+\tau \id)^{-1}V_i-\mathsf{c}_{\tau,p}|=\gO(\polyLog(n)/\sqrt{n})$. Recalling the notation 
$$
c_\tau=\frac{1}{n}\trace{\left[\frac{1}{n}\sum_{i=1}^n \psi'(R_i)X_iX_i\trsp+\tau \id_p\right]^{-1}}\;,
$$
which is the analog of $\mathsf{c}_{\tau,p}$ when we use all the predictors and not only $(p-1)$, we see that $\sup_i|c_i-c_\tau|=\gO(n^{-1/2}\polyLog(n))$.

\subsubsection{Control of $\xi_n$ and $\ansatzBetap$}
We can combine all the results we have obtained so far in the following proposition. 
\begin{proposition}\label{prop:explicitControlXin}
We have 
\begin{equation}\label{eq:explicitControlXin}
\left|\mathsf{c}_{\tau,p}(\xi_n+\tau)-\frac{p}{n}\right|\leq \gO_{L_k}\left((\sup_i \psi'(r_{i,[p]})\left(\frac{\polyLog(n)}{\sqrt{nc_n}}+\frac{1}{\tau^2} K_n \opnorm{\SigmaHat}+\frac{1}{n\tau}\right)\right)=\gO_{L_k}\left(\frac{\polyLog(n)}{\sqrt{n}}\right)\;.
\end{equation}
Furthermore, under our assumptions, 
\begin{equation}\label{eq:SecondMomentAnsatzBetap}
\left(\frac{p}{n}\right)^2 n \Exp{\ansatzBetap^2}=\frac{1}{n}\sum_{i=1}^n \Exp{(\mathsf{c}_{\tau,p} \psi(r_{i,[p]})^2}+\lo(1)\;.
\end{equation}
\end{proposition}
\begin{proof}
The proof of Equation \eqref{eq:explicitControlXin} consists just in aggregating all the previous results and noticing that $\mathsf{c}_{\tau,p}\leq p/(n\tau)$ and therefore remains bounded. 

We recall that
$$
(\tau+\xi_n)\sqrt{n}\ansatzBetap|\{V_i,\eps_i\}=\frac{1}{\sqrt{n}}\sum_{i=1}^n \psi(r_{i,[p]}) X_i(p)\;.
$$
Therefore,
$$
\mathsf{c}_{\tau,p}(\tau+\xi_n)\sqrt{n}\ansatzBetap|\{V_i,\eps_i\}=\frac{1}{\sqrt{n}}\sum_{i=1}^n \mathsf{c}_{\tau,p}\psi(r_{i,[p]}) X_i(p)
$$
Now, $\mathsf{c}_{\tau,p}\psi(r_{i,[p]})$, which depends only on $\{V_i,\eps_i\}_{i=1}^n$ is independent of $\{X_i(p)\}_{i=1}^n$. 

We conclude that 
$$
\Exp{(\mathsf{c}_{\tau,p}(\tau+\xi_n)\sqrt{n}\ansatzBetap)^2}=\frac{1}{n}\sum_{i=1}^n \Exp{(\mathsf{c}_{\tau,p} \psi(r_{i,[p]})^2}\;.
$$
Given the result in Equation \eqref{eq:explicitControlXin}, this means that 
$$
\left(\frac{p}{n}\right)^2 n \Exp{\ansatzBetap^2}=\frac{1}{n}\sum_{i=1}^n \Exp{(\mathsf{c}_{\tau,p} \psi(r_{i,[p]})^2}+\lo(1)\;.
$$
\end{proof}

\subsubsection{On $\mathsf{d}_{i,p}$}
Recall the definition 
$$
\mathsf{d}_{i,p}=[\psi'(\gamma^*_{i,p})-\psi'(r_{i,[p]})]\;,
$$
where $\gamma^*_{i,p} \in (r_{i,[p]},r_{i,[p]}+\nu_i)$, with 
$$
\nu_i=\ansatzBetap X_i\trsp \begin{bmatrix} (\mathfrak{S}_p+\tau \id)^{-1} u_p \\-1\end{bmatrix}=\ansatzBetap \pi_i\;.
$$
We call $\widetilde{B}_n(i)=\sup_i|r_{i,[p]}|+\sup_i |\pi_i|$.

We have the following result. 
\begin{proposition}\label{prop:controldip}
We have  
$$
\sup_i |\mathsf{d}_{i,p}|=\gO_{L_k}\left(\frac{\polyLog(n)}{\sqrt{n}c_n^{1/2}}L(\widetilde{B}_n(i)) \left[\psi'(-\widetilde{B}_n(i))\vee \psi'(\widetilde{B}_n(i))\right]\right)\;.
$$
Hence, 
$$
\sup_i |\mathsf{d}_{i,p}|=\gO_{L_k}\left(\frac{\polyLog(n)}{\sqrt{n}}\right)\;.
$$	
\end{proposition}

\begin{proof}
Recall the definition 
$$
\mathsf{d}_{i,p}=[\psi'(\gamma^*_{i,p})-\psi'(r_{i,[p]})]\;,
$$
where $\gamma^*_{i,p} \in (r_{i,[p]},r_{i,[p]}+\nu_i)$, with 
$$
\nu_i=\ansatzBetap X_i\trsp \begin{bmatrix} (\mathfrak{S}_p+\tau \id)^{-1} u_p \\-1\end{bmatrix}=\ansatzBetap \pi_i\;.
$$
Therefore, 
$$
\pi_i=V_i\trsp (\mathfrak{S}_p+\tau \id)^{-1} u_p-X_i(p)\;.
$$
Recall that $u_p=\frac{1}{n}V\trsp \DPsiPrimeOneLessPred X(p)$. According to Lemma \ref{lemma:controlRandomLipFuncConcRVsLp}, we have 
$$
\sup_i |V_i\trsp (\mathfrak{S}_p+\tau \id)^{-1} u_p|=\gO_{L_k}\left(\frac{\polyLog(n)}{c_n^{1/2}}\sup_i \norm{V_i\trsp (\mathfrak{S}_p+\tau \id)^{-1}\frac{1}{n}V\trsp \DPsiPrimeOneLessPred}\right))\;.
$$ 
Now, 
$$
\norm{V_i\trsp (\mathfrak{S}_p+\tau \id)^{-1}\frac{1}{n}V\trsp \DPsiPrimeOneLessPred}^2=\frac{1}{n}V_i\trsp (\mathfrak{S}_p+\tau \id)^{-1} \frac{V\trsp\DPsiPrimeOneLessPred^2 V}{n} (\mathfrak{S}_p+\tau \id)^{-1} V_i\;.
$$
Since $\mathfrak{S}_p=\frac{V\trsp\DPsiPrimeOneLessPred V}{n}$, we have $\frac{V\trsp\DPsiPrimeOneLessPred^2 V}{n}\preceq \opnorm{\DPsiPrimeOneLessPred}\mathfrak{S}_p$ and we conclude that 
$$
\frac{1}{n}V_i\trsp (\mathfrak{S}_p+\tau \id)^{-1} \frac{V\trsp\DPsiPrimeOneLessPred^2 V}{n} (\mathfrak{S}_p+\tau \id)^{-1} V_i\leq 
\frac{\norm{V_i}^2}{n\tau}\opnorm{\DPsiPrimeOneLessPred}=\frac{\norm{V_i}^2}{n\tau} \sup_{i}\psi'(r_{i,[p]})\;.
$$
We also note that $\sup_i X_i(p)=\gO_{L_k}(\polyLog(n)/\sqrt{c_n})$
and 
conclude that 
\begin{align*}
\sup_i |\pi_i|&=\gO_{L_k}\left(\frac{\polyLog(n)}{c_n^{1/2}}  \left[1+\sup_{i}\psi'(r_{i,[p]}) \sup_i \frac{\norm{V_i}^2}{n\tau}\right]\right)\;,\\
&=\gO_{L_k}\left(\frac{\polyLog(n)}{c_n^{1/2}}  \left[\sup_{i}\psi'(r_{i,[p]}) \right]\right)\;.
\end{align*}
Recalling that $\ansatzBetap=\gO_{L_k}(n^{-1/2}\polyLog(n))$, we finally see that 
$$
\sup_i \nu_i=\gO_{L_k}\left(\frac{\polyLog(n)}{\sqrt{n}c_n^{1/2}}  \left[\sup_{i}\psi'(r_{i,[p]}) \right]\right)
$$
As before, we can control $\sup_{i}\psi'(r_{i,[p]})$ by using the work done in Proposition \ref{prop:ControlBni}, since $r_{i,[p]}$ are the full residuals when we work with $p-1$ predictors. The growth conditions we have imposed on $\psi'$ and ${\mathcal E}_n$ therefore guarantee control of $\left[\sup_{i}\psi'(r_{i,[p]}) \right]$ in $L_k$.
Recall that $\widetilde{B}_n(i)=\sup_i|r_{i,[p]}|+\sup_i |\pi_i|$.
Now our assumptions guarantee that 
$$
\sup_i|\mathsf{d}_{i,p}|=\gO_{L_k}\left(\frac{\polyLog(n)}{\sqrt{n}c_n^{1/2}}L(\widetilde{B}_n(i)) \left[\psi'(-\widetilde{B}_n(i))\vee \psi'(\widetilde{B}_n(i))\right]\right)\;.
$$
Proposition \ref{prop:ControlBni} then allows us to conclude, by giving us $\polyLog$ bounds on $\widetilde{B}_n(i)$.
\end{proof}

\subsection{Final conclusions}
We finally have:
\begin{corollary}\label{coro:ApproxBetaHatThroughLeaveOnePredictorOut}
Assuming that $1/c_n=\gO(\polyLog(n))$, we have 
$$
\norm{\betaHat-\bTilde}\leq \frac{1}{\tau}\gO_{L_k}\left(\frac{\polyLog(n)}{n}\right)
$$	

In particular, 
\begin{align*}
\sqrt{n}(\betaHat_p-\ansatzBetap)&=\gO_{L_k}(\polyLog(n)/\sqrt{n})\;,\\
\sup_i |X_i\trsp (\betaHat-\bTilde)|&=\gO_{L_k}\left(\frac{\polyLog(n)}{\sqrt{n}}\right)\;,\\
\sup_{i}|R_i-r_{i,[p]}|&=\gO_{L_k}\left(\frac{\polyLog(n)}{\sqrt{n}}\right)\;.
\end{align*}

\end{corollary}
The corollary is just the aggregation of all of our results. 

The last statement is the only one that might need an explanation. With the notations of the proof of Proposition \ref{prop:controldip}, we have $R_i-r_{i,[p]}=X_i\trsp (\bTilde-\betaHat)+\nu_i$. The results in the proof of Proposition \ref{prop:controldip} as well as the bound on $\norm{\bTilde-\betaHat}$ give us the announced result. 

We note that when the vectors $X_i$'s are i.i.d with i.i.d entries, all the coordinates play a symmetric role, so using the results of the previous corollary, Equation \eqref{eq:SecondMomentAnsatzBetap} and summing over all the coordinates, we have, asymptotically,  
\begin{equation}\label{eq:firstStepTowardsSecondEquationOfSystem}
\frac{p}{n}\Exp{\norm{\betaHat}^2}=\frac{p^2}{n} \Exp{\ansatzBetap^2}+\lo(1)=\frac{1}{n}\sum_{i=1}^n \Exp{(\mathsf{c}_{\tau,p} \psi(r_{i,[p]})^2}+\lo(1)\;.
\end{equation}

\subsubsection{On $\mathsf{c}_{\tau,p}$ and $c_\tau$}
\begin{proposition}\label{prop:controldeltacctaup}
We have 
$$
|c_\tau-\mathsf{c}_{\tau,p}|=\gO_{L_k}(n^{-1/2}\polyLog(n))\;.
$$
\end{proposition}
\begin{proof}
Let us recall the notation 
$$
S=\frac{1}{n}\sum_{i=1}^n \psi'(R_i) X_i X_i\trsp\;.
$$
If we call $\Gamma=\frac{1}{n}\sum_{i=1}^n \psi'(R_i) V_i V_i\trsp$ and $a=\frac{1}{n}\sum_{i=1}^n\psi'(R_i)X_i^2(p)$, we see that 
$$
S=\begin{pmatrix}
\Gamma & \mathsf{v}\\
\mathsf{v}& a\;.
\end{pmatrix}\;.
$$
According to Lemma \ref{lemma:impactSizeAugmentationOnTrace}, we see that 
$$
|c_\tau-\frac{1}{n}\trace{(\Gamma+\tau \id)^{-1}}|\leq \frac{1}{n} \frac{1+a/\tau}{\tau}\;. 
$$
It is clear that under our assumptions, $a=\gO_{L_k}(\polyLog(n))$. It is also clear that 
$$
\sup_i |\psi'(R_i)-\psi'(r_{i,[p]})|=\gO_{L_k}(\polyLog(n)/\sqrt{n})\;.
$$
Hence, using arguments similar to the ones we have used in the proof of Lemma \ref{lemma:ControlEtais}, we see that 
$$
\left|\frac{1}{n}\trace{(\Gamma+\tau \id)^{-1}}-\frac{1}{n}\trace{(\mathfrak{S}_p+\tau \id)^{-1}}\right|=\gO_{L_k}(\polyLog(n)/\sqrt{n})\;.
$$
Since $\mathsf{c}_{\tau,p}=\frac{1}{n}\trace{(\mathfrak{S}_p+\tau \id)^{-1}}$, the result we announced follows immediately.
\end{proof}

In light of this result, we see that Equation \eqref{eq:firstStepTowardsSecondEquationOfSystem} can be re-written 
$$
\frac{p}{n}\Exp{\norm{\betaHat}^2}=\frac{1}{n}\sum_{i=1}^n \Exp{(c_\tau\psi(R_i))^2}+\lo(1)=\frac{1}{n}\sum_{i=1}^n \Exp{(c_i\psi(\prox_{c_i}(\rho)(\rTilde_{i,(i)})))^2}+\lo(1)\;,
$$
where we have used the remark we made after Lemma \ref{lemma:ControlKn} that showed that $\sup_i |c_i-c_\tau|=\gO_{L_k}(n^{-1/2}\polyLog(n))$. (See also Lemma \ref{lemma:DiffProxRespectc} and its proof where we compute the derivative of $\prox_c(\rho)(x)$ with respect to $c$.)

So we finally have 
\begin{equation}\label{eq:secondStepTowardsSecondEquationOfSystem}
\frac{p}{n}\Exp{\norm{\betaHat}^2}=\frac{1}{n}\sum_{i=1}^n \Exp{(c_\tau\psi(\prox_{c_\tau}(\rho)(\rTilde_{i,(i)})))^2}+\lo(1)\;.
\end{equation}
This will give us the second equation of our system. We also note that for any $x$, $c_\tau\psi[\prox_{c_\tau}(\rho)(x)]=x-\prox_{c_\tau}(\rho)(x)=\prox_1((c_\tau\rho)^*)(x)$ - see e.g \cite{MoreauProxPaper65}. In \cite{NEKOptimalMEstimationPNASPublished2013}, we found that this formulation was nicer when further analytic manipulations where needed.

\section{Putting things together}
\subsection{On the asymptotic distribution of $\rTilde_{i,(i)}$}
We have the following lemma.
\begin{lemma}
As $n$ and $p$ tend to infinity, $\rTilde_{i,(i)}$ behaves like $\eps_i+\sqrt{\Exp{\norm{\betaHat}^2}}Z$, where $Z\sim {\cal N}(0,1)$, in the sense of weak convergence. 

Furthermore, if $i\neq j$, $\rTilde_{i,(i)}$ and $\rTilde_{j,(j)}$ are asymptotically independent. 
\end{lemma}

\begin{proof}
The only problem is of course showing that 
$
\betaHat_{(i)}\trsp X_i
$
is approximately ${\cal N}(0,\Exp{\norm{\betaHat}^2})$. 
Recall that $\betaHat_{(i)}$ is independent of $X_i$. We assume without loss of generality that $\norm{\betaHat_{(i)}}$ remains bounded away from 0 in our asymptotics. Note that if it is not the case $\Exp{(\betaHat_{(i)}\trsp X_i)^2}=\Exp{\norm{\betaHat_{(i)}}^2}\tendsto 0$ and so $\betaHat_{(i)}\trsp X_i \WeakCv 0$, so the result holds.

Because $\var{\norm{\betaHat}^2}\tendsto 0$ and $\var{\norm{\betaHat_{(i)}}^2}\tendsto 0$, we see that 
$$
\frac{\norm{\betaHat_{(i)}}}{\Exp{\norm{\betaHat_{(i)}}}}\tendsto 1 \text{in probability}.
$$
Provided that we can apply the Lindeberg-Feller theorem (see e.g \cite{BreimanProbaBook}, p.186) conditional on a realization $X_{(i)}$, we will have 
$$
\frac{\betaHat_{(i)}\trsp X_i}{\norm{\betaHat_{(i)}}}|X_{(i)}\WeakCv {\cal N}(0,1)\;.
$$
Because the limit is independent of $\norm{\betaHat_{(i)}}$, we see that the result holds unconditionally, if we can apply the Lindeberg-Feller theorem with high $X_{(i)}$-probability.

And because $\norm{\betaHat_{(i)}}^2/\Exp{\norm{\betaHat}^2}\tendsto 1$ in probability, Slutsky's lemma allows us to conclude that under these assumptions we have 
$$
\frac{\betaHat_{(i)}\trsp X_i}{\sqrt{\Exp{\norm{\betaHat}^2}}}\WeakCv {\cal N}(0,1)
$$

The only question we have to check is therefore to verify that we can apply the Lindeberg-Feller theorem conditionally on $X_{(i)}$, at least with high $X_{(i)}$-probability. Recall that we have shown that 
$$
\betaHat_p=\gO_{L_k}(\frac{1}{\sqrt{n}\tau})\;.
$$ 
The same arguments we used apply also to $(\betaHat_{(i)})_p$. So it is clear that 
$$
\Exp{|(\betaHat_{(i)})_p|^3}=\gO(n^{-3/2})\;.
$$
We conclude that $\sum_{k=1}^n|(\betaHat_{(i)})_k|^3=\gO_{L_k, X_{(i)}}(n^{-1/2})\ll \norm{\betaHat_{(i)}}^2$ with high $X_{(i)}-$probability, since we are in the setting where $\norm{\betaHat_{(i)}}^2$ is bounded away from 0.

This shows the first part of the lemma.

For the second part, we use a leave-two-out approach, namely we use the approximation $\rTilde_{i,(i)}=\eps_i-\betaHat_{i}\trsp X_i= \eps_i-\betaHat_{(ij)}\trsp X_i+\gO_{L_k}(\polyLog(n)/(\sqrt{n}c_n))$ and similarly for $\rTilde_{j,(j)}$. It is clear that $\rTilde_{i,(i)}$ and $\rTilde_{j,(j)}$ are asymptotically independent conditional on $X_{(ij)}$. But because their dependence on $X_{(ij)}$ is only through $\norm{\betaHat_{(ij)}}$, which is asymptotically deterministic, we see that $\rTilde_{i,(i)}$ and $\rTilde_{j,(j)}$ are asymptotically independent. 

The lemma is shown.
\end{proof}
We are now in position to show that $c_\tau=\frac{1}{n}\trace{(S+\tau \id_p)^{-1}}$ is asymptotically deterministic and that the empirical distribution of the residuals $R_i$ is asymptotically non-random.

\begin{lemma}
Consider the random function 
$$
g_n(x)=\frac{1}{n}\sum_{i=1}^n \frac{1}{1+x\psi'(\prox_x(\rho)(\rTilde_{i,(i)}))}\;,
$$
defined for $x\geq 0$.
Let $B>0$ be in $\mathbb{R}_+$.
Call $F_{\rho,B}(u)=\left([\psi'(0)+L(|u|)|u|]+BL(|u|)[|\psi(u)|+|\psi(-u)|]\right)$, where $L(|u|)$ is the Lipschitz constant of $\psi'$ on $[-|u|,|u|]$.
We have, for any $(x,y)\in \mathbb{R}_+^2$, and $x\leq B$, $y\leq B$
$$
\sup_{(x,y):|x-y|\leq \eta, x\leq B, y\leq B}\left|g_n(x))-g_n(y)\right|\leq \eta \frac{1}{n}\sum_{i=1}^n F_{\rho,B}(\rTilde_{i,(i)})\;.
$$
In particular,  we have 
$$
P^*\left(\sup_{(x,y):|x-y|\leq \eta, x\leq B, y\leq B}\left|g_n(x))-g_n(y)\right|>\delta\right)\leq \frac{\eta}{\delta}
\frac{1}{n}\sum_{i=1}^n \Exp{F_{\rho,B}(\rTilde_{i,(i)})}\;.
$$
Hence, $g_n$ is stochastically equicontinuous on $[0,B]$ for any $B>0$ given, since under our assumptions $\Exp{F_{\rho,B}(\rTilde_{i,(i)})}$ is uniformly bounded in $n$,
\end{lemma}
We used the notation $P^*$ above to denote our probability and avoid a discussion of potential measure theoretic issues associated with taking a supremum over a non-countable collection of random variables. We refer the reader to e.g \cite{PollardConvergenceStochProcesses84} for more details on stochastic equicontinuity. 
\begin{proof}
Let us consider the function
$$
h_u(x)=\frac{1}{1+x\psi'(\prox_x(\rho)(u))}=\frac{\partial}{\partial u}\prox_x(\rho)(u)\;.
$$
The last equality comes from Lemma \ref{lemma:simpleObsProx}.

We have 
$$
\left|h_u(x)-h_u(y)\right|\leq |x\psi'(\prox_x(\rho)(u))-y\psi'(\prox_y(\rho)(u))|\wedge 1\;.
$$
Therefore, 
$$
\left|h_u(x)-h_u(y)\right|\leq |x-y|\psi'(\prox_x(\rho)(u))+y|\psi'(\prox_x(\rho)(u))-\psi'(\prox_y(\rho)(u))|\;.
$$
In particular, if $|x-y|\leq \eta$, and $x \vee y \leq B$
$$
\sup_{y:|x-y|\leq \eta}\left|h_u(x)-h_u(y)\right|\leq \eta \psi'(\prox_x(\rho)(u))+(x+\eta)\sup_{y}|\psi'(\prox_x(\rho)(u))-\psi'(\prox_y(\rho)(u))|\;.
$$
Note

Under our assumptions, Lemma \ref{lemma:ValueProxat0} implies that, for $y\geq 0$, $\sup_y|\prox_y(\rho)(u)|\leq |u|$. One of our assumptions is that $\psi'$ is Lipschitz on any $[-t,t]$ with Lipshitz constant $L(t)$. Therefore, 
$$
|\psi'(\prox_x(\rho)(u))-\psi'(\prox_y(\rho)(u))|\leq L(|u|) |\prox_x(\rho)(u)-\prox_y(\rho)(u)|\;.
$$
We recall that, according to Lemma \ref{lemma:DiffProxRespectc}, 
$$
\frac{\partial}{\partial x} \prox_x(\rho)(u)=-\frac{\psi(\prox_x(\rho)(u))}{1+c\psi'(\prox_x(\rho)(u))}\;.
$$
Furthermore, since $\psi$ is non-decreasing and changes sign at 0, we also have 
$$
\sup_x |\frac{\partial}{\partial x} \prox_x(\rho)(u)|\leq |\psi(u)|\vee |\psi(-u)|\;.
$$
This naturally gives us a bound on the Lipschitz constant of the function $x\rightarrow \prox_x(\rho)(u)$. We finally conclude that 
$$
|\psi'(\prox_x(\rho)(u))-\psi'(\prox_y(\rho)(u))|\leq L(|u|)[|\psi(u)|\vee |\psi(-u)|]|x-y|\;.
$$  
We therefore have 
$$
\sup_{y:|x-y|\leq \eta}\left|h_u(x)-h_u(y)\right|\leq \eta \psi'(\prox_x(\rho)(u)) +(x+\eta)L(|u|)[|\psi(u)|\vee |\psi(-u)|]\eta\;.
$$
Of course, $\psi'(\prox_x(\rho)(u))\leq \psi'(0)+L(|u|)|u|$, by using again the fact that $|\prox_x(\rho)(u)|\leq |u|$, the fact that $\prox_x(\rho)(0)=0$ and the fact that the Lipschitz constant of $\psi'$ on $[-|\prox_x(\rho)(u)|,|\prox_x(\rho)(u)|]$ is less than $L(u)$. 

Therefore, if $x+\eta\leq B$, we have 
$$
\sup_{y:|x-y|\leq \eta}\left|h_u(x)-h_u(y)\right|\leq \eta\left([\psi'(0)+L(|u|)|u|]+BL(|u|)[|\psi(u)|+|\psi(-u)|]\right)\;.
$$
Therefore, we also have 
$$
\sup_{(x,y):|x-y|\leq \eta, x\vee y\leq B}\left|h_u(x)-h_u(y)\right|\leq \eta\left([\psi'(0)+L(|u|)|u|]+BL(|u|)[|\psi(u)|+|\psi(-u)|]\right)\;.
$$
We denote by $F_{\rho,B}(u)=\left([\psi'(0)+L(|u|)|u|]+BL(|u|)[|\psi(u)|+|\psi(-u)|]\right)$\;.
This analysis shows that for $x$ given, if $|x-y|\leq \eta$ and $x\vee y \leq B$, we have 
$$
\sup_{(x,y):|x-y|\leq \eta, x\leq B, y\leq B}\left|g_n(x))-g_n(y)\right|\leq \eta \frac{1}{n}\sum_{i=1}^n F_{\rho,B}(\rTilde_{i,(i)})\;.
$$
We can now take expectations, and get the result in $L_1$ provided $\Exp{F_{\rho,B}(\rTilde_{i,(i)})}$ is finite and remains bounded in $n$. However, this holds since $F_{\rho,B}$ grows at most polynomially at $\infty$, and $\eps_i$, $\norm{\betaHat_{(i)}}$ and $X_i$ have infinitely many moments, by Assumptions \textbf{O4}, \textbf{F2} and our work on $\norm{\betaHat}$.

We have established stochastic equicontinuity of $g_n(x)$ on $[0,B]$. 
\end{proof}

\begin{lemma}\label{lemma:ControlSupDeltagnMeanGn}
Let us call $G_n(x)=\Exp{g_n(x)}$.
For any given $x_0\leq B$,
$$
g_n(x_0)-G_n(x_0)=\lo_{L_2}(1)\;.
$$ 

Under our assumptions, we also have 
$$
\outerExp{\sup_{0\leq x\leq B}|g_n(x)-G_n(x)|}\tendsto 0\;.
$$	
\end{lemma}

\begin{proof}
Asymptotic pairwise independence of $\rTilde_{i,(i)}$ implies that 
$$
\var{g_n(x_0)}\tendsto 0
$$
and therefore gives the first result. 
	
Let us pick $\eps>0$. By the stochastic equicontinuity of $g_n$ and our $L_1$ bound, we can find $x_1,\ldots,x_K$, independent of $n$, such that for all $x \in [0,B]$, there exists $l$ such that, when $n$ is large enough, 
$$
\Exp{|g_n(x)-g_n(x_l)|}\leq \eps\;.
$$
Note that 
$$
|g_n(x)-G_n(x)|\leq |g_n(x)-g_n(x_l)|+|g_n(x_l)-G_n(x_l)|+|G_n(x_l)-G_n(x)|\;.
$$
We immediately get 
$$
\outerExp{\sup_{0\leq x\leq B}|g_n(x)-G_n(x)|}\leq 2 \eps+\Exp{\sup_{1\leq l\leq K}|g_n(x_l)-G_n(x_l)|}\;.
$$
Because $K$ is finite, the fact that for all $l$, $|g_n(x_l)-G_n(x_l)|\tendsto 0$ in $L_2$ implies that $\sup_{1\leq l \leq K}|g_n(x_l)-G_n(x_l)|\tendsto 0$ in $L_2$. In particular, if $n$ is sufficiently large, 
$$
\Exp{\sup_{1\leq l\leq K}|g_n(x_l)-G_n(x_l)|}\leq \eps\;.
$$ 
The lemma is shown.
\end{proof}

\begin{lemma}
Call $c_\tau=\frac{1}{n}\trace{(S+\tau \id_p)^{-1}}$. 
Call as before
$$
g_n(x)=\frac{1}{n}\sum_{i=1}^n \frac{1}{1+x\psi'(\prox_x(\rho)(\rTilde_{i,(i)}))}
$$
Then $c_\tau$ is a near solution of 
$$
\frac{p}{n}-\tau x-1+g_n(x)=0\;,
$$
i.e $\frac{p}{n}-\tau c_\tau-1+g_n(c_\tau)=\lo_{L_k}(1)$.

Asymptotically, near solutions of 
$$
\delta_n(x)\triangleq \frac{p}{n}-\tau x-1+g_n(x)=0\;,
$$
are close to solutions of 
$$
\Delta_n(x)=\frac{p}{n}-\tau x-1+\Exp{g_n}(x)=0\;.
$$
More precisely, call $T_{n,\eps}=\{x:|\Delta_n(x)|<\eps\}$. Note that $T_{n,\eps}\subseteq (0,p/(n\tau)+\eps/\tau)$.
For any given $\eps$, as $n\tendsto \infty$, near solutions of $\delta_n(x_n)=0$ belong to $T_{n,\eps}$ with high-probability.

Our assumptions concerning the distribution of $\eps_i's$, specifically $\textbf{F1}$, guarantee that as $n\tendsto \infty$, there is a unique solution to $\Delta_n(x)=0$. 

Hence $c_\tau$ is asymptotically deterministic. 
\end{lemma}

\begin{proof}

Let $\delta_n$ be the function 
$$
\delta_n(x)=\frac{p}{n}-\tau x-1+g_n(x)\;,
$$	
and $\Delta_n(x)=\Exp{\Delta_n(x)}$. Call $x_n$ the solution $\delta_n(x_n)=0$ and $x_{n,0}$ the solution of $\Delta_n(x_{n,0})=0$.
Since $0\leq g_n \leq 1$, we see that $x_n\leq p/(n\tau)$, for otherwise, $\delta_n(x)<0$. The same argument shows that 
if $x>(p/n+\eps)/\tau$, $\Delta_n(x)<-\eps$ and $x\notin T_{n,\eps}$. Similarly, near solutions of $\delta_n(x)=0$ must be less or equal to $(p/n+\eps)/\tau$.

$\bullet$ {\small \textbf{Proof of the fact that $c_\tau$ is such that $\delta_n(c_\tau)=\lo(1)$}}\\
An important remark is that $c_\tau$ is a near solution of $\delta_n(x)=0$. This follows most clearly for arguments we have developed for $\mathsf{c}_{\tau,p}$ so we start by giving details through arguments for this random variable. Recall that in the notation of Lemma \ref{lemma:approxXinByTrace}, we had 
$$
p/n-\tau \mathsf{c}_{\tau,p}=\frac{1}{n}\trace{\id_n-M}.
$$
Now, according to Equation \eqref{eq:keyEqnInImplicitDefOfCtaup}, 
$$
\frac{1}{n}\trace{\id_n-M}=1-\frac{1}{n}\sum_{i=1}^n \frac{1}{1+\psi'(r_{i,[p]})\frac{1}{n}V_i\trsp (\mathfrak{S}_p(i)+\tau\id)^{-1}V_i}.
$$
According to Lemmas \ref{lemma:ControlEtais} and \ref{lemma:ControlKn}, we have 
$$
\sup_{i}\left|\frac{1}{n}V_i\trsp (\mathfrak{S}_p(i)+\tau\id)^{-1}V_i-\mathsf{c}_{\tau,p}\right|=\gO_{L_k}(\polyLog(n)n^{-1/2}).
$$
Of course, when $x\geq 0$ and $y\geq 0$, $|1/(1+x)-1/(1+y)|\leq |x-y|$. Using our bounds on $\psi'(r_{i,[p]})$, we easily see that, 
$$
p/n-\tau \mathsf{c}_{\tau,p}-1+\frac{1}{n}\sum_{i=1}^n \frac{1}{1+\mathsf{c}_{\tau,p}\psi'(r_{i,[p]})}=\gO_{L_k}(n^{-1/2}\polyLog(n))\;.
$$
Exactly the same computations can be made with $c_{\tau}$, so we have established that 
$$
p/n-\tau c_\tau-1+\frac{1}{n}\sum_{i=1}^n \frac{1}{1+c_{\tau}\psi'(R_i)}=\gO_{L_k}(n^{-1/2}\polyLog(n))\;.
$$
Now we have seen in Corollary \ref{coro:AggregResApproxBetaHatByBetaTildeIncl} that 
$$
\sup_{i}|R_i-\prox_{c_i}(\rho)(\rTilde_{i,i})|=\gO_{L_k}(n^{-1/2}\polyLog(n))\;.
$$
Through our assumptions on $\psi'$, this of course implies that 
$$
\sup_{i}|\psi'(R_i)-\psi'[\prox_{c_i}(\rho)(\rTilde_{i,i})]|=\gO_{L_k}(n^{-1/2}\polyLog(n))
$$
We have furthermore noted that $\sup_i|c_i-c_\tau|=\gO_{L_k}(n^{-1/2}\polyLog(n))$ after Lemma \ref{lemma:ControlKn}. Using the proof of Lemma \ref{lemma:DiffProxRespectc}, this implies that 
$$
\sup_i\left|\psi'[\prox_{c_i}(\rho)(\rTilde_{i,i})]-\psi'[\prox_{c_\tau}(\rho)(\rTilde_{i,i})]\right|=\gO_{L_k}(n^{-1/2}\polyLog(n))\;
$$
and therefore 
$$
\sup_i\left|\psi'[R_i]-\psi'[\prox_{c_\tau}(\rho)(\rTilde_{i,i})]\right|=\gO_{L_k}(n^{-1/2}\polyLog(n))\;
$$
So we have established that $\delta_n(c_\tau)=\gO_{L_k}(n^{-1/2}\polyLog(n))$.

$\bullet$ {\small \textbf{Final details}}\\
Note that for any given $x$, $\delta_n(x)-\Delta_n(x)=\lo_P(1)$ by using Lemma \ref{lemma:ControlSupDeltagnMeanGn}. In our case, with the notation of this lemma, $B=p/(n\tau)+\eta/\tau$, for $\eta>0$ given.

This implies that, for any given $\eps>0$ 
$$
\sup_{x\in(0,p/(n\tau)+\eta/\tau]}|\delta_n(x)-\Delta_n(x)|<\eps\;,
$$
with high-probability when $n$ is large. Therefore, for any $\eps>0$
$$
|\Delta_n(x_n)|\leq \eps
$$
with high-probability. 
This exactly means that $x_n \in T_{n,\eps}$ with high-probability. The same argument applies for near solutions of $\delta_n(x)=0$, which, for any $\eps>0$ must belong to $T_{n,\eps}$ as $n\tendsto \infty$ with high-probability. Of course, there is nothing random about $T_{n,\eps}$ which is a deterministic set. Note that $T_{n,\eps}$ is compact because it is bounded and closed, using the fact that $g_n$ and $\Exp{g_n}$ are continuous.

If $T_{n,0}$ were reduced to a single point, we would have established the asymptotically deterministic character of $c_\tau$. 

Given our work concerning the limiting behavior of $\rTilde_{i,(i)}$ and our assumptions about $\eps_i$'s, we see that Lemma \ref{lemma:suffCondKeyEqnInCHasUniqueSolution} applies to $\lim_{n\tendsto \infty}\Delta_n(x)$ under assumption \textbf{F1}. Therefore, as $n\tendsto \infty$, $T_{n,0}$ is reduced to a point and $c_\tau$ is asymptotically non-random.
\end{proof}
As we had noted in \cite{NEKetAlRobustRegressionTechReport11}, 
$$
\frac{\partial}{\partial t}\prox_c(\rho)(t)=\prox_c(\rho)'(t)=\frac{1}{1+c\psi'(\prox_c(\rho)(t))}\;.
$$
So $\Delta_n$ can be interpreted as 
$$
\Delta_n(x)=\frac{p}{n}-\tau x-1+\frac{1}{n}\sum_{i=1}^n \Exp{\prox_x(\rho)'(\rTilde_{i,(i)})}\;.
$$
The fact that $c_\tau$ is asymptotically arbitrarily close to the root of $\Delta_n(x)=0$ gives us the first equation in the system appearing in Theorem \ref{thm:systemRidgeRegularizedRigorous}.

The second equation of the system comes from Equation \eqref{eq:secondStepTowardsSecondEquationOfSystem}. Theorem \ref{thm:systemRidgeRegularizedRigorous} is shown.

\section{From the ridge-regularized to the un-regularized problem}\label{sec:regularizedToUnregularized}
Our original motivation in \cite{NEKetAlRobustRegressionTechReport11} was to study the unpenalized problem, namely $\betaHat$ was defined as 
$$
\betaHat=\argmin_{\beta} \frac{1}{n}\sum_{i=1}^n \rho(Y_i-X_i\trsp \beta)\;.
$$
We now explain how we can derive the system in the unpenalized case from the one we have obtained in the penalized case, when $p/n<1$. 

We first note that when $p<n$, and for instance the $X_i$'s have a continuous distribution, if $Y_i=X_i\trsp \beta_0+\eps_i$, 
$$
\betaHat-\beta_0=\argmin_{\beta}\frac{1}{n}\sum_{i=1}^n \rho(\eps_i-X_i\trsp \beta)\;.
$$
So to understand the error we make when using regression M-estimates, i.e the vector $\betaHat-\beta_0$, it is enough to study the properties of our estimator in the null case where $\beta_0=0$. Of course, we have previously studied the penalized version of this particular problem. 

We have the following theorem. 
\begin{theorem}
Suppose our assumptions \textbf{O1-O7}, \textbf{P1} and \textbf{F1-F2} hold. Suppose further that $\limsup p/n<1$. 
Call 
\begin{align*}
\betaHat&=\argmin_{\beta}\frac{1}{n}\sum_{i=1}^n \rho(\eps_i-X_i\trsp \beta)\;,\text{ and }\\
\betaHat_\tau&=\argmin_{\beta}\frac{1}{n}\sum_{i=1}^n \rho(\eps_i-X_i\trsp \beta)+\tau \frac{\norm{\betaHat}^2}{2}\;.
\end{align*}

If $\rho$ is strongly convex, 
$$
\lim_{\tau \tendsto 0} \norm{\betaHat_{\tau}-\betaHat}=\lo_P(1)\;.
$$
Hence $\norm{\betaHat}$ is asymptotically deterministic and can be computed via 
$$
\norm{\betaHat}=\lim_{\tau\tendsto 0} \norm{\betaHat_\tau}\;.
$$
\end{theorem}
\begin{proof}
We call $f_{\tau}(\beta)=-\frac{1}{n}\sum_{i=1}^n X_i\psi(\eps_i-X_i\trsp \beta)+\tau \beta$ and $f(\beta)=f_0(\beta)$. 

Since by definition, $\betaHat$ is such that
$$
\sum X_i \psi(\eps_i-X_i\trsp \betaHat)=0\;,
$$
we see that $f_{\tau}(\betaHat)=\tau \betaHat$. By a similar token, we see that $f_0(\betaHat_\tau)=-\tau \betaHat_\tau$. If $\rho$ is strongly convex with modulus of convexity $C$, we see, using Proposition \ref{prop:ControlDeltaBetaDeltaf} that, by working with $\nabla f_\tau$, we get
$$
\norm{\betaHat_\tau-\betaHat}\leq \frac{\tau}{C \lambda_{\min}(\SigmaHat)+\tau} \norm{\betaHat}\;,
$$
and by working with $\nabla f_0$ - along the same lines as in the proof of Proposition \ref{prop:ControlDeltaBetaDeltaf} - we get 
$$
\norm{\betaHat_\tau-\betaHat}\leq \frac{\tau}{C \lambda_{\min}(\SigmaHat)} \norm{\betaHat_\tau}\;.
$$
Recall that we showed in Equation \eqref{eq:boundNormBetaHatFromSimplyRidge} that 
$$
\norm{\betaHat_\tau}\leq \sqrt{\frac{2}{\tau}} \sqrt{\frac{1}{n}\sum_{i=1}^n \rho(\eps_i)}\;.
$$
This shows that 
$$
\norm{\betaHat_\tau-\betaHat}\leq \frac{\sqrt{2\tau}}{C \lambda_{\min}(\SigmaHat)} \sqrt{\frac{1}{n}\sum_{i=1}^n \rho(\eps_i)}\;.
$$
Under our assumptions, $\frac{1}{n}\sum_{i=1}^n \rho(\eps_i)=\gO_P(1)$. 
Under the assumptions that, for instance, the entries of $X_i$'s are i.i.d with 4+$\eps$ moments (which is always the case under our assumptions), it is well known that $\lambda_{\min}(\SigmaHat)\tendsto (1-\sqrt{\frac{p}{n}})^2$ in probability and a.s (\cite{bai99}).

We conclude that $\norm{\betaHat_\tau-\betaHat}\tendsto 0$ in probability as $\tau \tendsto 0$ under our assumptions. 
\end{proof}

Under for instance Gaussian design assumptions (i.e $X_i$'s have distribution ${\cal N}(0,\id_p)$), it is possible to bound $\Exp{1/\lambda_{\min}(\SigmaHat)}$ using essentially results in \cite{silverstein85} as well as elementary but non-trivial linear algebra (see the appendix of \cite{troppmartinsson2011} for instance). This would give an approximation in $L_2$, provided the random variable $\rho(\eps_i)$ has enough moment. 

It would be possible with quite a bit of extra work to dispense with the assumption of strong convexity and move for instance to strict convexity (see \cite{HiriartLemarechalConvexAnalysisAbridged2001} if needed). We refer to the discussion after the proof of Proposition  \ref{prop:ControlDeltaBetaDeltaf} for more details about this issue. 

We note that convergence in probability of $\betaHat$ is enough for our confidence interval statements from \cite{NEKOptimalMEstimationPNASPublished2013} (details in the supplementary material of that paper) to go through. 

\subsection{Other extensions}
Moving from random vectors $X_i$'s like the ones we have studied to vectors of the form $\tilde{X}_i=\lambda_i X_i$, where $\lambda_i$ is a random variable (i.e a scalar) independent of $X_i$ does not offer any new conceptual difficulties. Indeed our heuristic work in \cite{NEKRobustPaperPNAS2013Published} handled - heuristically of course -  that case, so the arguments we gave here would be easy to modify. This extended class of models - which is akin to elliptical distributions in multivariate statistics (see \cite{anderson03}) - is interesting because it includes distributions that do not share the geometric properties that ``concentrated'' random vectors share. We do not solve the elliptical problem here in complete details because of the extra notational burden involved. 

Another easy extension of the work presented here is to study the weighted regression case, i.e for weights $\{w_i\}_{i=1}^n$, $\betaHat$ is defined as 
$$
\betaHat=\argmin_{\beta \in \mathbb{R}^p} \sum_{i=1}^n w_i \rho(\eps_i-X_i\trsp \beta)+\frac{\tau}{2}\norm{\beta}^2\;.
$$
Once again, only minor modifications are needed to our proof - the heuristic we proposed easily handled this. More generally, working on the problem of understanding 
$$
\betaHat=\argmin_{\beta \in \mathbb{R}^p} \sum_{i=1}^n \rho_i(\eps_i-X_i\trsp \beta)+\frac{\tau}{2}\norm{\beta}^2\;,
$$
where $\rho_i$ are potentially different functions and $X_i$'s are ``elliptical'' (as defined above) seems to be within relatively easy reach of the method developed and presented here. 

Finally, we see that when $Y_i=X_i\trsp \beta_0+\eps_i$, 
$$
\betaHat_\tau-\beta_0=\argmin_{\beta \in \mathbb{R}^p} \frac{1}{n}\sum_{i=1}^n \rho(\eps_i-X_i\trsp \beta)+\frac{\tau}{2} \norm{\beta+\beta_0}^2\;.
$$
This problem - a mild variant of the one we have studied here - should be amenable to analysis with the method we used here. 

\newpage
\appendix
\begin{center}
\textbf{\textsc{APPENDIX}}
\end{center}
\vspace{1cm}
\renewcommand{\thesection}{\Alph{section}}
\renewcommand{\theequation}{\Alph{section}-\arabic{equation}}
\renewcommand{\thelemma}{\Alph{section}-\arabic{lemma}}
\renewcommand{\thecorollary}{\Alph{section}-\arabic{corollary}}
\renewcommand{\thesubsection}{\Alph{section}-\arabic{subsection}}
\renewcommand{\thesubsubsection}{\Alph{section}-\arabic{subsection}.\arabic{subsubsection}}
  % redefine the command that creates the equation no.
\setcounter{equation}{0}  % reset counter
\setcounter{lemma}{0}
\setcounter{corollary}{0}
\setcounter{section}{0}
\section{Notes on the prox}
\begin{lemma}\label{lemma:ValueProxat0}
Let $\rho$ be such that $\psi$ changes sign at 0, i.e $\sgn(\psi(x))=\sgn(x)$ for $x\neq 0$. Then,
$$
\prox_c(\rho)(0)=0\;.
$$	
Furthermore,
$$
\left|\psi(\prox_c(\rho)(x))\right|\leq |\psi(x)|\;.
$$
Also,
$$
\left|\psi(\prox_c(\rho)(x))\right|\leq |x|/c\;.
$$
\end{lemma}	
\begin{proof}
By definition, we have
$$
\prox_c(\rho)(x)+c\psi(\prox_c(\rho)(x))=x\;.
$$	
Therefore,
$$
\prox_c(\rho)(0)=-c\psi(\prox_c(\rho)(0))\;.
$$
Hence, if we call $y=\prox_c(\rho)(0)$, we have $\sgn(y)=-\sgn(\psi(y))$. The assumptions on $\psi$ therefore guarantee that $y=0$, for otherwise we would have a contradiction.

We also note that $\sgn(\prox_c(\rho)(x))=\sgn(x)$, since $\sgn[\psi(\prox_c(\rho)(x))]=\sgn(\prox_c(\rho)(x)))$.

Using contractivity of the prox, we see that
$$
|\prox_c(\rho)(x)|=|\prox_c(\rho)(x)-\prox_c(\rho)(0)|\leq |x|\;.
$$
Since $\rho$ is convex, we see that $\psi$ is increasing. If $x>0$, $\prox_c(\rho)(x)>0$, and therefore,
$$
0\leq \psi(\prox_c(\rho)(x))\leq \psi(x)\;.
$$
Similarly, if $x<0$, $x\leq \prox_c(\rho)(x)<0$ and therefore, $\psi(x)\leq \psi(\prox_c(\rho)(x))<0$. The second statement of the lemma is shown.

The last statement is a simple consequence of the fact that $c\psi(\prox_c(\rho)(x))=x-\prox_c(\rho)(x)$, from which it immediately follows that
$$
|\psi(\prox_c(\rho)(x))|\leq \frac{|x|}{c}\;.
$$
\end{proof}

We will also need the following simple result. 

\begin{lemma}\label{lemma:DiffProxRespectc}
Suppose $x$ is a real and $\rho$ is twice differentiable and convex. Then, for $c>0$,
$$
\frac{\partial}{\partial c}\rho(\prox_c(\rho)(x))=-\frac{\psi^2(\prox_c(\rho)(x))}{1+c\psi'(\prox_c(\rho)(x))}\;.
$$	
In partial, at $x$ given $c\rightarrow \rho(\prox_c(\rho)(x))$ is decreasing in $c$.
\end{lemma}
\newcommand{\pcx}{\prox_c(\rho)(x)}
\begin{proof}
Using the fact that 
$$
\pcx+c\psi(\pcx)=x,
$$
we easily see that 
$$
\frac{\partial}{\partial c}\pcx=-\frac{\psi(\prox_c(\rho)(x))}{1+c\psi'(\prox_c(\rho)(x))}\;.
$$
It then follows immediately that 
$$
\frac{\partial}{\partial c}\rho(\prox_c(\rho)(x))=-\frac{\psi^2(\prox_c(\rho)(x))}{1+c\psi'(\prox_c(\rho)(x))}\;.
$$
The denominator is positive, from which we immediately deduce that $c\rightarrow \rho(\prox_c(\rho)(x))$ is decreasing in $c$.
\end{proof}
We also make the following observation, which was essential to finding the system of equations in \cite{NEKRobustPaperPNAS2013Published}
\begin{lemma}\label{lemma:simpleObsProx}
We have 
$$
\frac{\partial}{\partial x}\pcx=\frac{1}{1+c\psi'(\pcx)}\;.
$$
\end{lemma}
A proof of this fact follows immediately from the well-known representation (see \cite{MoreauProxPaper65})
$$
\pcx=(\id+c\psi)^{-1}(x)\;.
$$
We finally make notice of the following simple fact. 
\begin{lemma}\label{lemma:simpleObsxpsiproxsquared}
The function $x\rightarrow [x\psi(\prox_x(\rho)(u))]^2$ (defined on $\mathbb{R}_+$) is increasing, for any $u$.
\end{lemma}

\begin{proof}
Let us consider $f_u(x)=x\psi(\prox_x(\rho)(u))$. Note that $f_u(x)=u-\prox_x(\rho)(u)$. So 
$$
\frac{\partial}{\partial x} f_u(x)=\frac{\psi(\prox_x(\rho)(u))}{1+x\psi'(\prox_x(\rho(u)))}\;.
$$
Hence, 
$$
\frac{\partial}{\partial x} f^2_u(x)=2x\frac{\psi^2(\prox_x(\rho)(u))}{1+x\psi'(\prox_x(\rho(u)))}\geq 0
$$
since $x\geq 0$ and $\psi'\geq 0$.
\end{proof}

\section{On convex Lipschitz functions of random variables}
In this section, we provide a brief reminder concerning convex Lipschitz functions of random variables.
\begin{lemma}\label{lemma:controlLipFuncConcRVsLp}
Suppose that $\{X_i\} \in \mathbb{R}^p$ satisfy the following concentration property: $\exists C_n, c_n$ such that for any $F_i$, a convex, 1-Lipschitz function of $X_i$, 
$$
P(|F_i(X_i)-m_i|\geq t)\leq C_n \exp(-c_n t^2)\;,
$$
where $m_i$ is deterministic. 
Then if ${\cal F}_n=\sup_i |F_i(X_i)-m_i|$, we have, even when the $X_i$'s are dependent:
\begin{enumerate}
\item if $u_n=\sqrt{\log(n)/c_n}$, $\Exp{{\cal F}_n}\leq u_n+C_n/(2\sqrt{c_n}\sqrt{\log n})=\frac{\sqrt{\log n}}{\sqrt{c_n}}\left(1+C_n/(2\log n)\right)$. Similar bounds hold in $L_k$ for any finite given $k$.
\item when $C_n\leq C$, where $C$ is independent of $n$, there exists $K$, independent of $n$ such that ${\cal F}_n/u_n \leq K$ with overwhelming probability, i.e probability asymptotically smaller than any power of $1/n$. 
\end{enumerate}
In particular, 
$$
{\cal F}_n=\gO(\polyLog(n)/\sqrt{c_n})
$$
in probability and any $L_k$, $k$ fixed and given.
\end{lemma}
We note that similar techniques can be used to extend the result to situations where we have 
$P(|F_i(X_i)-m_i|\geq t)\leq C_n \exp(-c_n t^\alpha)$, with $\alpha\neq 2$. Of course, the order of magnitudes of the bounds then change.
\begin{proof}
Clearly, 
$$
P({\cal F}_n\geq t)\leq 1\wedge n C_n \exp(-c_n t^2)\;.
$$
Hence, for any $u\geq 0$, 
$$
\Exp{{\cal F}_n^k}\leq u^k+\int_u^{\infty} k t^{k-1} n C_n \exp(-c_n t^2)\;,
$$
since $\Exp{{\cal F}_n^k}=\int_0^{\infty} k t^{k-1} P({\cal F}_n\geq t) dt$. Standard computations show that when $uc_n^2$ is large, and $k\geq 1$,
$$
\int_u^{\infty} t^{k-1}  \exp(-c_n t^2)dt \sim \gO(\frac{u^{k}}{2c_n u^2}\exp(-c_n u^2))\;.
$$
So we see that in that case, for a constant $k$ that depends only on $k$,
$$
\Exp{{\cal F}^k_n}\leq u^k (1+K_k\frac{n C_n}{c_n u^2}\exp(-c_n u^2))\;.
$$
Taking $u_n=\sqrt{\log n/c_n}$, we see that 
$$
\Exp{{\cal F}^k_n}\leq u_n^k(1+K_k\frac{C_n}{\log n})\;.
$$
We conclude that when $C_n/\log n$ remains bounded, $\Exp{{\cal F}^k_n}/u_n^k$ remains bounded. In the case $k=1$, it is easy to see that $K_k=1/2$ and we do not require $\sqrt{c_n}u$ to be large for our arguments to go through. This gives the bound announced in the Lemma.

The probabilistic bound comes simply from the fact that 
$$
P({\cal F}_n\geq t u_n)\leq C_n \exp(-\log_n (t^2-1))\leq C_n \exp(-\log_n (t^2-1))\;.
$$
Hence, 
$$
P\left(\frac{{\cal F}_n}{u_n}\geq K\right)\leq n^{-d}\;,
$$
for any given $d$ if $K$ is large enough.
If we allow $K$ to grow like a power of $\log n$, we also see that the right hand side above can be made even smaller.

\end{proof}
We recall that we denote by $X_{(i)}=\{X_1,\ldots,X_{i-1},X_{i+1},\ldots,X_n\}$. If $I$ is a subset of $\{1,\ldots,n\}$ of size $n-1$, we call $X_I$ the collection of the corresponding $X_i$ random variables. We call $X_{I^c}$ the remaining random variable. 
\begin{lemma}\label{lemma:controlRandomLipFuncConcRVsLp}
Suppose $X_i$'s are independent and satisfy the concentration inequalities as above. 
Consider the situation where $F_{I_k}$ is a convex Lipschitz function of 1 variable, depending on $X_{I_k}$ only and call ${\cal L}_{I_k}$ its Lipschitz constant (at $X_{I_k}$ given)- which is assumed to be random. Call $m_{F_{I_k}}=m_{F_i(X_{I_k^c})}|X_{I_k}$, $m$ being the mean or the median. As before, call ${\mathcal F}_n=\sup_{j=1,\ldots,n}|F_{I_j}(X_{I_j^c})-m_{F_{I_j}}|$
Then ${\mathcal F_n}=\gO(\sqrt{\log n/c_n}\sup_{1\leq j\leq n} {\cal L}_{I_j})$ in probability and in $\sqrt{L_{2k}}$, i.e there exists $K>0$, independent of $n$, such that 
$$
\Exp{{\mathcal F_n}^k}\leq K (\sqrt{\log n/c_n})^k \sqrt{\Exp{\sup_{1\leq j\leq n} {\cal L}_{I_j}^{2k}}}\;.
$$
Hence, ${\mathcal F_n}$ is $\polyLog(n)/c_n^{1/2} \sup_{1\leq j\leq n} {\cal L}_{I_j}$ in $\sqrt{L_{2k}}$.
\end{lemma}
\begin{proof}
We call ${\cal L}=\sup_i {\cal L}_i$.
By Holder's inequality, we have 
$$
\Exp{{\mathcal F_n}^k}=\Exp{({\mathcal F_n}^k/{\cal L}^k) {\cal L}^k}\leq \sqrt{\Exp{{\mathcal F_n}^{2k}/{\cal L}^{2k}}} \sqrt{\Exp{{\cal L}^{2k}}}\;.
$$
Let us call $\widetilde{\mathcal F}_n={\mathcal F_n}/{\cal L}$. As before, 
\begin{align*}
\Exp{\widetilde{\mathcal F}_n^k}&\leq u^k+\sum_{j=1}^n\int_{u}^{\infty}k x^{k-1}P(|F_{I_j}(X_{I_j^c})-m_{F_{I_j}}|\geq {\cal L} x)\;,\\
&\leq u^k+\sum_{j=1}^n\int_{u}^{\infty}k x^{k-1}P(|F_{I_j}(X_{I_j^c})-m_{F_{I_j}}|\geq L_{I_j} x)\;,\\
&=u^k+\sum_{j=1}^n\int_{u}^{\infty}k x^{k-1}\Exp{P(|F_{I_j}(X_{I_j^c})-m_{F_{I_j}}|\geq L_{I_j} x|X_{I_j})}\;.
\end{align*}
Now our assumptions guarantee that 
$$
P(|F_{I_k}(X_{I_k^c})-m_{F_{I_k}}|\geq L_{I_k} x|X_{I_k})\leq C_n \exp(-c_n x^2)\;,
$$
since $F_{I_k}/L_{I_k}$ is 1-Lipschitz (and independent of $X_{I_k^c}$).
We conclude that 
$$
\Exp{\widetilde{\mathcal F}_n^k}\leq u^k+n C_n \int_{u}^{\infty} k x^{k-1}\exp(-c_n x^2)\;.
$$
This is exactly the same situation as we had before and the conclusion follows. 
\end{proof}

\begin{lemma}\label{lemma:RandomQuadFormsBounds}
Suppose the assumptions of the previous Lemma are satisfied. 
Consider $Q_{I_j}=\frac{1}{n}X_{I_j^c}\trsp M_{I_j} X_{I_j^c}$, where $M$ is a random symmetric matrix depending only on $X_{I_j}$ whose largest eigenvalue is $\lambda_{max,I_j}$. Assume that $\scov{X_i}=\id_p$ and $nc_n\tendsto \infty$.
Then, we have in $L_k$, 
$$
\sup_{1\leq j \leq n} \left|Q_{I_j}-\frac{1}{n}\trace{M_{I_j}}\right|=\gO_{L_k}(\frac{\polyLog(n)}{\sqrt{nc_n}}\sup_{1\leq j\leq n} \lambda_{max,I_j})\;.
$$
The same bound holds when considering a single $Q_{I_j}$ without the $\polyLog(n)$ term.
\end{lemma}
\begin{proof}
Lemma \ref{lemma:controlRandomLipFuncConcRVsLp} applies to $\sqrt{Q_{I_j}}$ and $\sup_{1\leq j \leq n} |\sqrt{Q_{I_j}}-m_{\sqrt{Q_{I_j}}}|$. The corresponding 
Lipschitz constant if of course $\sqrt{\lambda_{max,I_j}/n}$.

So all we need to do is show that we can go from this control to the control of $\sup_{1\leq j \leq n} |Q_{I_j}-\frac{1}{n}\trace{M_{I_j}}|$.

Of course, 
$$
|Q_{I_j}-\frac{1}{n}\trace{M_{I_j}}|\leq |Q_{I_j}-m^2_{\sqrt{Q_{I_j}}}|+|m^2_{\sqrt{Q_{I_j}}}-\frac{1}{n}\trace{M_{I_j}}|\;.
$$
The idea from there is simply to use the fact that for $a$ and $b$ non-negative, $(a+b)^{k}\leq 2^{k-1}(a^k+b^k)$.
Using Proposition 1.9 in \citet{ledoux2001}, we know that 
$$
|m^2_{\sqrt{Q_{I_j}}}-\frac{1}{n}\trace{M_{I_j}}|\leq \frac{C_n}{nc_n}\lambda_{max}(M_{I_j})\;.
$$
On the other hand, 
$$
|Q_{I_j}-m^2_{\sqrt{Q_{I_j}}}|=\left|\sqrt{Q_{I_j}}-m_{\sqrt{Q_{I_j}}}\right|\left|\sqrt{Q_{I_j}}+m_{\sqrt{Q_{I_j}}}\right|\leq \left|\sqrt{Q_{I_j}}-m_{\sqrt{Q_{I_j}}}\right|^2+
2\left|\sqrt{Q_{I_j}}-m_{\sqrt{Q_{I_j}}}\right| m_{\sqrt{Q_{I_j}}}\;,
$$
since $m_{\sqrt{Q_{I_j}}}\geq 0$.

Therefore, 
$$
\sup_{1\leq j \leq n} |Q_{I_j}-m^2_{\sqrt{Q_{I_j}}}|\leq \sup_{1\leq j \leq n}\left|\sqrt{Q_{I_j}}-m_{\sqrt{Q_{I_j}}}\right|^2+2\left[\sup_{1\leq j \leq n}
\left|\sqrt{Q_{I_j}}-m_{\sqrt{Q_{I_j}}}\right|\right] \left[\sup_{1\leq j \leq n}  m_{\sqrt{Q_{I_j}}}\right]
$$

Lemma \ref{lemma:controlRandomLipFuncConcRVsLp} gives us control of $\sup_{1\leq j \leq n}\left|\sqrt{Q_{I_j}}-m_{\sqrt{Q_{I_j}}}\right|$  in $L_{2k}$ and therefore control of 
$\sup_{1\leq j \leq n}\left|\sqrt{Q_{I_j}}-m_{\sqrt{Q_{I_j}}}\right|^2$ in $L_{2k}$  with a bound of the form $\frac{\polyLog(n)}{(nc_n)}\sup_{1\leq j \leq n} \lambda_{max}(M_{I_j})$.

The result will therefore be shown provided we control $\left[\sup_{1\leq j \leq n}
\left|\sqrt{Q_{I_j}}-m_{\sqrt{Q_{I_j}}}\right|\right] \left[\sup_{1\leq j \leq n}  m_{\sqrt{Q_{I_j}}}\right]$. By using Holder's inequality and our control of 
$\left[\sup_{1\leq j \leq n}
\left|\sqrt{Q_{I_j}}-m_{\sqrt{Q_{I_j}}}\right|\right]$ in $L_{2k}$, it is clear that the only issue remaining is control of $\left[\sup_{1\leq j \leq n}  m_{\sqrt{Q_{I_j}}}\right]$ in $L_{2k}$.

Since $m_{\sqrt{Q_{I_j}}}=\Expj{X_{I_j}}{\sqrt{X_{I_j^c}\trsp M_{I_j} X_{I_j^c}/n}}\leq \sqrt{\Expj{X_{I_j}}{X_{I_j^c}\trsp M_{I_j} X_{I_j^c}/n}}=\sqrt{\trace{M_{I_j}}/n}$, since $\scov{X_{I_j}}=\id_p$, we see that 
$$
\left[\sup_{1\leq j \leq n}  m_{\sqrt{Q_{I_j}}}\right]\leq \sqrt{p/n} \sup_{1\leq j \leq n} \sqrt{\lambda_{\max,I_j}}\;.
$$
Therefore, 
$$
\left[\sup_{1\leq j \leq n}
\left|\sqrt{Q_{I_j}}-m_{\sqrt{Q_{I_j}}}\right|\right] \left[\sup_{1\leq j \leq n}  m_{\sqrt{Q_{I_j}}}\right]\leq K\frac{\polyLog(n)\sqrt{p/n}}{\sqrt{nc_n}}\sup_{1\leq j \leq n}\lambda_{\max,I_j}\text{ in } L_k\;,
$$
provided all the random quantities we work with have $2k$ moments. 

The conclusions of the Lemma follow by recalling our assumption that $p/n$ remains bounded and using the fact that $1/c_n\geq K/\sqrt{c_n}$ in the situations we are considering, i.e $c_n$ bounded but possibly going to zero. 
\end{proof}

\begin{lemma}\label{lemma:controlApproxQuadFormSDepXi}
Suppose $S_i=\frac{1}{n}\sum_{j\neq i}^n d_j X_j X_j\trsp$ where $\{d_j\}_{j=1}^n$ depends on $\{X_i\}_{i=1}^n$. Suppose we can find $\{\tilde{d}_{j,(i)}\}_{j\neq i}$ independent of $X_i$ such that $\sup_{j\neq i}|d_j-\tilde{d}_{j,(i)}|\leq \delta_n(i)$. Then we have 
$$
\left|\frac{1}{n}X_i\trsp (S_i+\tau \id_p)^{-1}X_i-\frac{1}{n} \trace{(S_i+\tau \id_p)^{-1}}\right|
=\gO_{L_k}\left(\delta_n(i)\opnorm{\SigmaHat}\vee \frac{1}{nc_n\tau}\right)\;.
$$

If the same can be done with all $i$'s and $\sup_{1\leq i\leq n}\delta_n(i)\leq K_n$, we have 
$$
\sup_i\left|\frac{1}{n}X_i\trsp (S_i+\tau \id_p)^{-1}X_i-\frac{1}{n} \trace{(S_i+\tau \id_p)^{-1}}\right|=\gO_{L_k}\left(
(1+\frac{\polyLog(n)}{\sqrt{nc_n}})K_n\opnorm{\SigmaHat}\vee \frac{\polyLog(n)}{\sqrt{nc_n}\tau}
\right)
$$
\end{lemma}
\begin{proof}
We call $\widetilde{S}_i=\frac{1}{n}\sum_{j\neq i} \tilde{d}_{j,(i)} X_j X_j\trsp$. 
Recall the first resolvent identity, i.e $A^{-1}-B^{-1}=A^{-1}(B-A)B^{-1}$. 
Clearly, by applying this identity of $S_i+\tau \id_p$ and $\widetilde{S}_i+\tau \id_p$, we have 
$$
\left|\frac{1}{n}X_i\trsp (S_i+\tau \id_p)^{-1}X_i-\frac{1}{n}X_i\trsp (\widetilde{S}_i+\tau \id_p)^{-1}X_i\right|\leq \frac{\norm{X_i}^2}{n}\frac{1}{\tau^2}\delta_n(i) \opnorm{\SigmaHat}\;,
$$
by using the fact that $\opnorm{\frac{1}{n}\sum_{j\neq i}(\tilde{d}_{j,(i)}-d_j) X_jX_j\trsp}\leq \sup_{j\neq i}|\tilde{d}_{j,(i)}-d_j|\opnorm{\SigmaHat}$.
	
Since $\widetilde{S}_i$ is independent of $X_i$, we can apply the results from Lemma \ref{lemma:RandomQuadFormsBounds} to see that 
$$
\frac{1}{n}X_i\trsp (\widetilde{S}_i+\tau \id_p)^{-1}X_i-\frac{1}{n}\trace{(\widetilde{S}_i+\tau \id_p)^{-1}}=\gO_{L_k}(\frac{1}{\sqrt{nc_n}\tau})\;.
$$
On the other hand, by making use again of the first resolvent identity, we see that 
\begin{align*}
\left|\frac{1}{n}\trace{(S_i+\tau \id_p)^{-1}-(\widetilde{S}_i+\tau \id_p)^{-1}}\right|&=\left|\frac{1}{n}\trace{(S_i+\tau \id_p)^{-1}\frac{1}{n}\sum_{j\neq i}(\tilde{d}_{j,(i)}-d_j) X_jX_j\trsp (\widetilde{S}_i+\tau \id_p)^{-1}}\right|\;,\\
&=\left|\frac{1}{n^2}\sum_{j\neq i} (\tilde{d}_{j,(i)}-d_j) X_j\trsp (\widetilde{S}_i+\tau \id_p)^{-1}(S_i+\tau \id_p)^{-1} X_j\right|\;,\\
&\leq \frac{1}{n^2}\sum_{j\neq i} \left|\tilde{d}_{j,(i)}-d_j\right|\frac{\norm{X_j}^2}{\tau^2}\;,\\
&\leq \delta_n(i) \frac{1}{n}\sum_{i=1}^n \frac{\norm{X_j}^2}{n \tau^2}
\end{align*}
We conclude that 
$$
\left|\frac{1}{n}X_i\trsp (S_i+\tau \id_p)^{-1}X_i-\frac{1}{n} \trace{(S_i+\tau \id_p)^{-1}}\right|=\gO_{L_k}\left( 
\frac{\norm{X_i}^2}{n}\frac{1}{\tau^2}\delta_n(i) \opnorm{\SigmaHat}+\frac{1}{\sqrt{nc_n} \tau}+\delta_n(i)\frac{1}{n}\sum_{i=1}^n \frac{\norm{X_j}^2}{n \tau^2}\right)\;.
$$

The result for the $\sup$ follows by the same technique and adjusting for taking $\sup$ in the various approximations. 
\end{proof}
\subsection*{On the spectral norm of covariance matrices}
In this subsection, we show that under our initial concentration assumptions, we can control $\opnorm{\SigmaHat}$. These results are very likely known but we did not find a reference covering precisely the same question we consider. The proof is a simple adaption of the well-known $\eps$-net argument explained e.g in \cite{TalagrandSpinGlassesBook03}, Appendix A.4.

\begin{lemma}\label{lemma:controlSpectralNorms}
Suppose $X_i$'s are independent random vectors in $\mathbb{R}^p$, satisfying our concentration assumptions in \textbf{O4}, and having mean 0 and covariance $\id_p$. Let $\SigmaHat=\frac{1}{n}\sum_{i=1}^n X_i X_i\trsp$. Then,  
$$
\norm{\SigmaHat}=\gO_P(c_n^{-1/2})\;. 
$$
The results hold also in $L_k$.
\end{lemma}

\begin{proof}
We study the largest singular value, $\sigma_1$ of the matrix $X/\sqrt{n}$, where the $i$-th row of $X$ is $X_i$. Of course, 
$$
\sigma_1(X/\sqrt{n})=\sup_{u,v,\norm{u}=1,\norm{v}=1} \frac{1}{\sqrt{n}}u\trsp X v\;.
$$
Note that 
$$
u\trsp Xv=\sum_{i=1}^n u_i (X_i\trsp v)\;.
$$
Consider first the case where $c_n=1$. Under our assumptions, $X_i\trsp v$ are independent subGaussian random vectors, with mean 0. Note that $\var{X_i\trsp v}=1$ if $\scov{X_i}=1$ and $\norm{v}=1$. Computing the moment generating function of $u\trsp X v$, we see that this random variable is itself subGaussian and has variance 1. Therefore, we have for all $t$, and constants $c_1$ and $c_2$, 
$$
P(|u\trsp X v|>t)\leq c_1 \exp(-c_2 t^2)\;.
$$
The $\eps$-net argument given in the proof of Lemma A.4.1 in \cite{TalagrandSpinGlassesBook03} then can be applied and the conclusions of that Lemma reached. (A slight adaption is needed to handle the fact that $u\in \mathbb{R}^n$ and $v\in \mathbb{R}^p$ but it is completely trivial and omitted). The fact that the results hold in $L_k$ is a simple consequence of the proof. 

In the case where $c_n\neq 1$, we just need to note that the moment generating function of $u\trsp X v$ is smaller than that of a Gaussian random variable with variance $1/c_n$. The result follows immediately.   
\end{proof}
\section{Miscellaneous results}
\subsection{An analytic remark}
One of our assumptions concerns the existence and uniqueness of a solution of the equation $F(x)=0$, where
$$
F(x)=\frac{p}{n}-\tau x -1 +\Exp{(\prox_x(\rho))'(W)}
$$
where $W$ is a random variable and $(\prox_x(\rho))'(t)=\frac{\partial }{\partial t} \prox_x(\rho)(t)=\frac{1}{1+x\psi'(\prox_x(\rho)(t))}$.

We now show that under mild conditions on $W$ this equation has a unique solution. This guarantees that our assumptions are not terribly strong and in particular apply to problems of interest to statisticians.

\begin{lemma}\label{lemma:suffCondKeyEqnInCHasUniqueSolution}
Suppose that $W$ has a smooth density $f$ with $\sgn(f'(x))=-\sgn(x)$. Suppose further that $\lim_{|t|\tendsto \infty} tf(t)=0$ and that $\sgn(\psi(x))=\sgn(x)$. Then, if 
$$
F(x)=\frac{p}{n}-\tau x -1 +\Exp{(\prox_x(\rho))'(W)}\;,
$$
the equation $F(x)=0$ has a unique solution. 
\end{lemma}

\begin{proof}
We call 
$$
G(x)\triangleq \Exp{(\prox_x(\rho))'(W)}\;.
$$
Of course, 
$$
\Exp{(\prox_x(\rho))'(W)}=\int (\prox_x(\rho))'(t) f(t) dt.
$$
Using contractivity of the proximal mapping (see \cite{MoreauProxPaper65}) we see that $\lim_{|t|\tendsto \infty} \prox_x(\rho)(t)f(t)=0$ under our assumptions. 

Integrating the previous equation by parts, we see that 
$$
\Exp{(\prox_x(\rho))'(W)}=-\int (\prox_x(\rho))(t) f'(t) dt\;.
$$
To compute $G'(x)$, we differentiate under the integral sign (under our assumptions the conditions of Theorem 9.1 in \cite{durrett96} are satisfied) to get 
$$
G'(x)=\int \frac{\psi(\prox_x(\rho)(t)) f'(t)}{1+x\psi'(\prox_x(\rho)(t))}dt\;.
$$
Under our assumptions, $\sgn(\psi(\prox_x(\rho)(t)))=\sgn(t)$ and $\sgn(f'(t))=-\sgn(t)$, so that 
$$
\forall t\neq 0, \sgn(\psi(\prox_x(\rho)(t)) f'(t))=-1 \;.
$$
Since the denominator of the function we integrate is positive, we conclude that 
$$
G'(x)\leq 0\;.
$$
Since $F'(x)=-\tau+G'(x)$, we see that $F'(x)<0$. Therefore $F$ is a decreasing function on $\mathbb{R}_+$.
Of course, $F(0)=p/n$ and $\lim_{x\tendsto \infty} F(x)=-\infty$. So we conclude that the equation $F(x)=0$ has a unique root.
\end{proof}
\textbf{Remark:} the conditions on the density of $W$ are satisfied in many situations. For instance if $W=\eps+rZ$, where $\eps$ is symmetric about 0 and log-concave,  $Z$ is ${\cal N}(0,1)$ and $r>0$, it is clear that the density of $W$ satisfies the conditions of our lemma. Similar results hold under weaker assumptions on $\eps$ of course but since the paper is already a bit long, we do not dwell on these issues which are well-known in the theory of log-concave functions (see e.g \cite{KarlinTotalPositivity68}, \cite{PrekopaLogConcaveSzeged73} and \cite{IbragimovLogConcave1956}). 
\subsection{A linear algebraic remark}
We have the following lemma. 
\begin{lemma}\label{lemma:impactSizeAugmentationOnTrace}
Suppose the $p\times p$ matrix $A$ is positive semi-definite and 
$$
A=\begin{pmatrix} 
\Gamma & v \\
v\trsp& a
\end{pmatrix}\;.
$$
Here $a\in \mathbb{R}$. 
Let $\tau$ be a strictly positive real. Call $\Gamma_\tau=\Gamma+\tau \id_{p-1}$. 
Then we have 
$$
\trace{(A+\tau)^{-1}}=\trace{\Gamma_\tau^{-1}}+\frac{1+v\trsp \Gamma_\tau^{-2}v}{a+\tau-v\trsp\Gamma_\tau^{-1}v}\;.
$$
In particular, 
$$
\left|\trace{(A+\tau)^{-1}}-\trace{\Gamma_\tau^{-1}}\right|\leq \frac{1+a/\tau}{\tau}\;.
$$
\end{lemma}

\begin{proof}
The first equation is simply an application of the block inversion formula for matrices (see \cite{hj}, p.18) and the Sherman-Morrison-Woodbury formula (\cite{hj}, p.19). 
Suppose temporarily that $A$ is positive definite. Then the Schur complement formula guarantees that $a\geq v\trsp \Gamma^{-1}v > v\trsp \Gamma_\tau^{-1} v$. The fact that $a\geq v\trsp \Gamma_{\tau}^{-1}v$ in general is obtained by a continuity argument (change $A$ to $A_\eps=A+\eps \id_p$ and let $\eps$ tend to 0). This implies that 
$$
\frac{1}{a+\tau-v\trsp \Gamma_{\tau}^{-1}v}\leq \frac{1}{\tau}\;.
$$
Since $v\trsp \Gamma_\tau^{-2}v\leq \frac{1}{\tau}v\trsp \Gamma_\tau^{-1}v\leq a/\tau$, we get the second equation.
\end{proof}

\bibliographystyle{/Users/nkaroui/Documents/NekTexAuxiliaries/Bibliography/annstats}
\bibliography{/Users/nkaroui/Documents/NekTexAuxiliaries/Bibliography/research}

\def\cprime{$'$}
\begin{thebibliography}{38}
\expandafter\ifx\csname natexlab\endcsname\relax\def\natexlab#1{#1}\fi
\expandafter\ifx\csname url\endcsname\relax
  \def\url#1{\texttt{#1}}\fi
\expandafter\ifx\csname urlprefix\endcsname\relax\def\urlprefix{URL }\fi

\bibitem[{Anderson(2003)}]{anderson03}
\textsc{Anderson}, T.~W. (2003).
\newblock \emph{An introduction to multivariate statistical analysis}.
\newblock Wiley Series in Probability and Statistics. Wiley-Interscience [John
  Wiley \& Sons], Hoboken, NJ, third edition.

\bibitem[{Bai(1999)}]{bai99}
\textsc{Bai}, Z.~D. (1999).
\newblock Methodologies in spectral analysis of large-dimensional random
  matrices, a review.
\newblock \emph{Statist. Sinica} \textbf{9}, 611--677.
\newblock With comments by G. J.\ Rodgers and Jack W.\ Silverstein; and a
  rejoinder by the author.

\bibitem[{Baranchik(1973)}]{baranchikInadmissibility73}
\textsc{Baranchik}, A.~J. (1973).
\newblock Inadmissibility of maximum likelihood estimators in some multiple
  regression problems with three or more independent variables.
\newblock \emph{Ann. Statist.} \textbf{1}, 312--321.

\bibitem[{Bayati and Montanari(2012)}]{BayatiMontanariLASSORISKIEEE2012}
\textsc{Bayati}, M. and \textsc{Montanari}, A. (2012).
\newblock The {LASSO} risk for {G}aussian matrices.
\newblock \emph{IEEE Trans. Inform. Theory} \textbf{58}, 1997--2017.
\newblock \urlprefix\url{http://dx.doi.org/10.1109/TIT.2011.2174612}.

\bibitem[{Bean et~al.(2013)Bean, Bickel, El~Karoui, and
  Yu}]{NEKOptimalMEstimationPNASPublished2013}
\textsc{Bean}, D., \textsc{Bickel}, P.~J., \textsc{El~Karoui}, N., and
  \textsc{Yu}, B. (2013).
\newblock Optimal m-estimation in high-dimensional regression.
\newblock \emph{Proceedings of the National Academy of Sciences} \textbf{110},
  14563--14568.
\newblock \urlprefix\url{http://www.pnas.org/content/110/36/14563.abstract}.

\bibitem[{Beck and Teboulle(2010)}]{BeckAndTeboulleChapter2010}
\textsc{Beck}, A. and \textsc{Teboulle}, M. (2010).
\newblock \emph{Convex Optimization in Signal Processing and Communications},
  chapter Gradient-Based Algorithms with Applications in Signal Recovery
  Problems, pp. 33--88.
\newblock Cambridge University Press.

\bibitem[{Breiman(1992)}]{BreimanProbaBook}
\textsc{Breiman}, L. (1992).
\newblock \emph{Probability}, volume~7 of \emph{Classics in Applied
  Mathematics}.
\newblock Society for Industrial and Applied Mathematics (SIAM), Philadelphia,
  PA.
\newblock Corrected reprint of the 1968 original.

\bibitem[{Donoho and Montanari(2013)}]{DonohoMontanariRobustArxiv13}
\textsc{Donoho}, D. and \textsc{Montanari}, A. (2013).
\newblock High dimensional robust m-estimation: Asymptotic variance via
  approximate message passing.
\newblock \emph{arXiv:1310.7320} .

\bibitem[{Donoho et~al.(2009)Donoho, Maleki, and
  Montanari}]{DonohoMalekiMontanariAMP09PNAS}
\textsc{Donoho}, D.~L., \textsc{Maleki}, A., and \textsc{Montanari}, A. (2009).
\newblock Message-passing algorithms for compressed sensing.
\newblock \emph{Proceedings of the National Academy of Sciences} \textbf{106},
  18914--18919.
\newblock \urlprefix\url{http://www.pnas.org/content/106/45/18914.abstract}.

\bibitem[{Durrett(1996)}]{durrett96}
\textsc{Durrett}, R. (1996).
\newblock \emph{Probability: theory and examples}.
\newblock Duxbury Press, Belmont, CA, second edition.

\bibitem[{Efron and Stein(1981)}]{EfronStein81}
\textsc{Efron}, B. and \textsc{Stein}, C. (1981).
\newblock The jackknife estimate of variance.
\newblock \emph{Ann. Statist.} \textbf{9}, 586--596.

\bibitem[{{El Karoui}(2009)}]{nekCorrEllipD}
\textsc{{El Karoui}}, N. (2009).
\newblock Concentration of measure and spectra of random matrices: Applications
  to correlation matrices, elliptical distributions and beyond.
\newblock \emph{The Annals of Applied Probability} \textbf{19}, 2362--2405.

\bibitem[{{El Karoui} et~al.(2012){El Karoui}, Bean, Bickel, Lim, and
  Yu}]{NEKetAlRobustRegressionTechReport11}
\textsc{{El Karoui}}, N., \textsc{Bean}, D., \textsc{Bickel}, P., \textsc{Lim},
  C., and \textsc{Yu}, B. (2012).
\newblock On robust regression with high-dimensional predictors.
\newblock Technical report.

\bibitem[{El~Karoui et~al.(2013)El~Karoui, Bean, Bickel, Lim, and
  Yu}]{NEKRobustPaperPNAS2013Published}
\textsc{El~Karoui}, N., \textsc{Bean}, D., \textsc{Bickel}, P.~J.,
  \textsc{Lim}, C., and \textsc{Yu}, B. (2013).
\newblock On robust regression with high-dimensional predictors.
\newblock \emph{Proceedings of the National Academy of Sciences}
  \urlprefix\url{http://www.pnas.org/content/early/2013/08/15/1307842110.abstract}.

\bibitem[{Halko et~al.(2011)Halko, Martinsson, and Tropp}]{troppmartinsson2011}
\textsc{Halko}, N., \textsc{Martinsson}, P.~G., and \textsc{Tropp}, J.~A.
  (2011).
\newblock Finding structure with randomness: probabilistic algorithms for
  constructing approximate matrix decompositions.
\newblock \emph{SIAM Rev.} \textbf{53}, 217--288.
\newblock \urlprefix\url{http://dx.doi.org/10.1137/090771806}.

\bibitem[{Hiriart-Urruty and
  Lemar{\'e}chal(2001)}]{HiriartLemarechalConvexAnalysisAbridged2001}
\textsc{Hiriart-Urruty}, J.-B. and \textsc{Lemar{\'e}chal}, C. (2001).
\newblock \emph{Fundamentals of convex analysis}.
\newblock Grundlehren Text Editions. Springer-Verlag, Berlin.
\newblock Abridged version of {{\i}t Convex analysis and minimization
  algorithms. I} [Springer, Berlin, 1993; MR1261420 (95m:90001)] and {{\i}t II}
  [ibid.; MR1295240 (95m:90002)].

\bibitem[{Horn and Johnson(1990)}]{hj}
\textsc{Horn}, R.~A. and \textsc{Johnson}, C.~R. (1990).
\newblock \emph{Matrix analysis}.
\newblock Cambridge University Press, Cambridge.
\newblock Corrected reprint of the 1985 original.

\bibitem[{Huber(1972)}]{HuberWaldLecture72}
\textsc{Huber}, P.~J. (1972).
\newblock The 1972 {W}ald lecture. {R}obust statistics: {A} review.
\newblock \emph{Ann. Math. Statist.} \textbf{43}, 1041--1067.

\bibitem[{Huber(1973)}]{HuberRobustRegressionAsymptoticsETCAoS73}
\textsc{Huber}, P.~J. (1973).
\newblock Robust regression: asymptotics, conjectures and {M}onte {C}arlo.
\newblock \emph{Ann. Statist.} \textbf{1}, 799--821.

\bibitem[{Huber and Ronchetti(2009)}]{HuberRonchettiRobustStatistics09}
\textsc{Huber}, P.~J. and \textsc{Ronchetti}, E.~M. (2009).
\newblock \emph{Robust statistics}.
\newblock Wiley Series in Probability and Statistics. John Wiley \& Sons Inc.,
  Hoboken, NJ, second edition.
\newblock \urlprefix\url{http://dx.doi.org/10.1002/9780470434697}.

\bibitem[{Ibragimov(1956)}]{IbragimovLogConcave1956}
\textsc{Ibragimov}, I.~A. (1956).
\newblock On the composition of unimodal distributions.
\newblock \emph{Teor. Veroyatnost. i Primenen.} \textbf{1}, 283--288.

\bibitem[{Karlin(1968)}]{KarlinTotalPositivity68}
\textsc{Karlin}, S. (1968).
\newblock \emph{Total positivity. {V}ol. {I}}.
\newblock Stanford University Press, Stanford, Calif.

\bibitem[{Ledoux(2001)}]{ledoux2001}
\textsc{Ledoux}, M. (2001).
\newblock \emph{The concentration of measure phenomenon}, volume~89 of
  \emph{Mathematical Surveys and Monographs}.
\newblock American Mathematical Society, Providence, RI.

\bibitem[{Mammen(1989)}]{MammenRobustRegressionAos89}
\textsc{Mammen}, E. (1989).
\newblock Asymptotics with increasing dimension for robust regression with
  applications to the bootstrap.
\newblock \emph{Ann. Statist.} \textbf{17}, 382--400.
\newblock \urlprefix\url{http://dx.doi.org/10.1214/aos/1176347023}.

\bibitem[{Mar{\v{c}}enko and Pastur(1967)}]{mp67}
\textsc{Mar{\v{c}}enko}, V.~A. and \textsc{Pastur}, L.~A. (1967).
\newblock Distribution of eigenvalues in certain sets of random matrices.
\newblock \emph{Mat. Sb. (N.S.)} \textbf{72 (114)}, 507--536.

\bibitem[{Moreau(1965)}]{MoreauProxPaper65}
\textsc{Moreau}, J.-J. (1965).
\newblock Proximit\'e et dualit\'e dans un espace hilbertien.
\newblock \emph{Bull. Soc. Math. France} \textbf{93}, 273--299.

\bibitem[{Pollard(1984)}]{PollardConvergenceStochProcesses84}
\textsc{Pollard}, D. (1984).
\newblock \emph{Convergence of stochastic processes}.
\newblock Springer Series in Statistics. Springer-Verlag, New York.

\bibitem[{Portnoy(1984)}]{PortnoyMestLargishPNConsistencyAoS84}
\textsc{Portnoy}, S. (1984).
\newblock Asymptotic behavior of {$M$}-estimators of {$p$} regression
  parameters when {$p^{2}/n$} is large. {I}. {C}onsistency.
\newblock \emph{Ann. Statist.} \textbf{12}, 1298--1309.
\newblock \urlprefix\url{http://dx.doi.org/10.1214/aos/1176346793}.

\bibitem[{Portnoy(1985)}]{PortnoyMestLargishPNCLTAoS85}
\textsc{Portnoy}, S. (1985).
\newblock Asymptotic behavior of {$M$} estimators of {$p$} regression
  parameters when {$p^2/n$} is large. {II}. {N}ormal approximation.
\newblock \emph{Ann. Statist.} \textbf{13}, 1403--1417.
\newblock \urlprefix\url{http://dx.doi.org/10.1214/aos/1176349744}.

\bibitem[{Portnoy(1987)}]{PortnoyCLTRobustRegressionJMVA87}
\textsc{Portnoy}, S. (1987).
\newblock A central limit theorem applicable to robust regression estimators.
\newblock \emph{J. Multivariate Anal.} \textbf{22}, 24--50.
\newblock \urlprefix\url{http://dx.doi.org/10.1016/0047-259X(87)90073-X}.

\bibitem[{Pr{\'e}kopa(1973)}]{PrekopaLogConcaveSzeged73}
\textsc{Pr{\'e}kopa}, A. (1973).
\newblock On logarithmic concave measures and functions.
\newblock \emph{Acta Sci. Math. (Szeged)} \textbf{34}, 335--343.

\bibitem[{Rangan(2011)}]{Rangan2011}
\textsc{Rangan}, S. (2011).
\newblock Generalized approximate message passing for estimation with random
  linear mixing.
\newblock In \emph{IEEE International Symp. On Information Theory (St.
  Petersburg)}.

\bibitem[{Shcherbina and Tirozzi(2003)}]{ShcherbinaTirozzi03}
\textsc{Shcherbina}, M. and \textsc{Tirozzi}, B. (2003).
\newblock Rigorous solution of the {G}ardner problem.
\newblock \emph{Comm. Math. Phys.} \textbf{234}, 383--422.
\newblock \urlprefix\url{http://dx.doi.org/10.1007/s00220-002-0783-3}.

\bibitem[{Silverstein(1985)}]{silverstein85}
\textsc{Silverstein}, J.~W. (1985).
\newblock The smallest eigenvalue of a large-dimensional {W}ishart matrix.
\newblock \emph{Ann. Probab.} \textbf{13}, 1364--1368.

\bibitem[{Silverstein(1995)}]{silverstein95}
\textsc{Silverstein}, J.~W. (1995).
\newblock Strong convergence of the empirical distribution of eigenvalues of
  large-dimensional random matrices.
\newblock \emph{J. Multivariate Anal.} \textbf{55}, 331--339.

\bibitem[{Stein(1960)}]{SteinInadmissibilityRegression60}
\textsc{Stein}, C. (1960).
\newblock Multiple regression.
\newblock In \emph{Contributions to probability and statistics}, pp. 424--443.
  Stanford Univ. Press, Stanford, Calif.

\bibitem[{Talagrand(2003)}]{TalagrandSpinGlassesBook03}
\textsc{Talagrand}, M. (2003).
\newblock \emph{Spin glasses: a challenge for mathematicians}, volume~46 of
  \emph{Ergebnisse der Mathematik und ihrer Grenzgebiete. 3. Folge. A Series of
  Modern Surveys in Mathematics [Results in Mathematics and Related Areas. 3rd
  Series. A Series of Modern Surveys in Mathematics]}.
\newblock Springer-Verlag, Berlin.
\newblock Cavity and mean field models.

\bibitem[{Wachter(1978)}]{wachter78}
\textsc{Wachter}, K.~W. (1978).
\newblock The strong limits of random matrix spectra for sample matrices of
  independent elements.
\newblock \emph{Annals of Probability} \textbf{6}, 1--18.

\end{thebibliography}

\end{document}